\documentclass[11pt]{article}
\usepackage[a4paper,margin=1in]{geometry}
\usepackage{amsmath,amssymb,amsthm,mathtools,bm}
\usepackage{microtype}
\usepackage{enumitem}
\usepackage[colorlinks=true,linkcolor=blue,citecolor=blue,urlcolor=blue]{hyperref}
\usepackage[nameinlink,capitalize]{cleveref}

\newtheorem{theorem}{Theorem}[section]
\newtheorem{proposition}[theorem]{Proposition}
\newtheorem{lemma}[theorem]{Lemma}

\theoremstyle{remark}
\newtheorem{remark}[theorem]{Remark}

\newcommand{\R}{\mathbb R}
\newcommand{\C}{\mathbb C}
\newcommand{\N}{\mathbb N}
\newcommand{\Z}{\mathbb Z}
\newcommand{\T}{\mathbb T}

\newcommand{\Om}{\Omega}
\newcommand{\eps}{\varepsilon}
\newcommand{\ip}[2]{\left\langle #1,#2\right\rangle}
\newcommand{\norm}[1]{\left\lVert #1\right\rVert}
\newcommand{\abs}[1]{\left\lvert #1\right\rvert}
\newcommand{\dd}{\,\mathrm d}
\newcommand{\supp}{\operatorname{supp}}
\newcommand{\Lip}{\mathrm{Lip}}
\newcommand{\Id}{\mathrm{Id}}
\newcommand{\HS}{\mathrm{HS}}

\newcommand{\Span}{\operatorname{span}}

\title{Almost-sure quenched KAM tori for cubic NLS
with a spatial white-noise potential}
\author{
Yingdu Dong\thanks{Affiliation: School of Mathematical Sciences, Fudan University, Shanghai, China. Email: 202131130009@mail.bnu.edu.cn.}
\and
Wenwen Jian\thanks{Affiliation: School of Mathematics, Physics and Statistics,
Shanghai Polytechnic University, Shanghai, China. Email: wwjian@sspu.edu.cn. Funding: NNSFC(11201392).}
\and
Xiaoping Yuan\thanks{Affiliation: School of Mathematical Sciences, Fudan University, Shanghai, China. Email: xpyuan@fudan.edu.cn. Funding: NNSFC(12371189).}
}
\date{}

\begin{document}
\maketitle

\begin{abstract}
Let
\[
 A_\omega=-\partial_x^2+\rho\dot B_x(\omega),
 \qquad \rho\ne0,
\]
be the Dirichlet Schr\"odinger operator on $(0,\pi)$ with a spatial
Gaussian white-noise potential, realized pathwise via
quasi-derivatives.  Here $\dot B_x$ denotes the distributional derivative
of Brownian motion with respect to the spatial variable $x$.  For $\kappa\ne0$, we consider the cubic nonlinear
Schr\"odinger equation
\[
 i u_t=A_\omega u+\kappa |u|^2u.
\]
We first prove a zero-set theorem for locally real-analytic functions on
classical Wiener space.  As a consequence, on a single event of probability
one, no nontrivial finitely supported integer combination of the
random eigenvalues vanishes, and the quartic twist matrix is nonsingular for every
finite tangential set.  After fixing a path in this event, we combine these qualitative nondegeneracy properties with a partial quartic Birkhoff normal form adapted to the KAM scheme. For every finite nonempty tangential
set \(J\) of cardinality \(b\), we  obtain a Cantor family of
linearly stable, real-analytic, small-amplitude invariant \(b\)-tori. The
family is parametrized by a Cantor subset of \([\nu,2\nu]^b\) whose relative
measure tends to one as \(\nu\to0\).
\end{abstract}



\tableofcontents

\section{Introduction}
We study the one-dimensional nonlinear Schr\"odinger equation
\begin{equation}\label{eq:NLS}
 i u_t=A_\omega u+\kappa |u|^2u,
 \qquad
 A_\omega=-\partial_x^2+\rho\dot B_x(\omega),
 \qquad 0<x<\pi,
\end{equation}
with Dirichlet boundary conditions.  Here $B=(B_x)_{0\le x\le\pi}$ is a standard real Brownian motion on a probability space $(\Omega,\mathcal F,\mathbb P)$, and $\rho,\kappa\in\R\setminus\{0\}$.  Throughout, $\dot B_x=\partial_x B_x$ denotes the derivative in the sense of distributions with respect to the spatial variable $x$ for each fixed $\omega$.  For every continuous Brownian path, the potential belongs to $H^{-1}(0,\pi)$, and hence the operator $A_\omega$ is interpreted through the primitive $\sigma_\omega=\rho B(\omega)$ and the quasi-derivative realization.

The operator $A_\omega$ describes a particle confined to a disordered
one-dimensional medium.  The kinetic term is $-\partial_x^2$, while the real potential represents spatial fluctuations of the medium.  Gaussian
white noise is a natural idealization when the correlation length of
these fluctuations is small compared with the observation scale.  As
discussed in \cite{DongJianYuan2026}, consider a centered stationary
random field $V$ with integrable covariance
$R(r)=\mathbb E[V(r)V(0)]$ and the rescaling
\[
 V_\epsilon(x)=\epsilon^{-1/2}V(x/\epsilon).
\]
Its covariance satisfies, in the distributional sense,
\[
 \mathbb E[V_\epsilon(x)V_\epsilon(y)]
 =\epsilon^{-1}R((x-y)/\epsilon)
 \longrightarrow \rho_{\mathrm{eff}}^2\delta(x-y),
 \qquad
 \rho_{\mathrm{eff}}^2=\int_\R R(r)\dd r.
\]
Under suitable mixing and moment assumptions, a functional central limit
theorem also gives
\[
 \int_0^x V_\epsilon(s)\dd s
 \Longrightarrow \rho_{\mathrm{eff}}B_x
\]
on compact intervals, and hence convergence of $V_\epsilon$ to spatial
white noise as a random distribution.  This scaling interpretation
motivates the choice $\rho\dot B$ in \eqref{eq:NLS}, with $\rho^2$ measuring the disorder strength.  On $(0,\pi)$ the potential is characterized by
\[
 \mathbb E\bigl[\langle\rho\dot B,f\rangle
                    \langle\rho\dot B,g\rangle\bigr]
 =\rho^2\int_0^\pi f(x)g(x)\dd x,
 \qquad f,g\in C_c^\infty(0,\pi;\R).
\]
Thus disorder is delta-correlated in space, and the Dirichlet boundary conditions describe confinement to the finite interval.  

The linear operator belongs to the one-dimensional continuous Anderson
model.  Early rigorous studies of white-noise Schr\"odinger operators include the work of Fukushima and Nakao~\cite{FukushimaNakao1977}; more recent work establishes
the construction and localization theory on the line
\cite{DumazLabbe2024}.  Here the interval is fixed, and the spectral
information needed for the nonlinear problem concerns arbitrarily high
Dirichlet modes.  Since Brownian paths are continuous, the primitive
$\sigma_\omega=\rho B(\omega)$ belongs to $L^2(0,\pi)$ almost surely.
The quasi-derivative theory for $H^{-1}$ potentials consequently gives a
pathwise self-adjoint realization of $A_\omega$ with compact resolvent and
form domain $H^1_0(0,\pi)$
\cite{SavchukShkalikov2003,HrynivMykytyuk2006,HrynivMykytyuk2012}.
The high-energy eigenvalue and eigenfunction estimates used below were
obtained in our preliminary work \cite{DongJianYuan2026}.

Infinite-dimensional KAM theory for Hamiltonian PDEs began in the late
1980s with Kuksin's persistence results for infinite-dimensional linear and
integrable Hamiltonian systems and their PDE applications
\cite{Kuksin1987,Kuksin1989Izvestiya,Kuksin1989Sbornik}.  In the early
parameter-dependent formulation, the tangential and normal frequencies vary
with a finite-dimensional external vector, and the values for which the
Melnikov conditions fail are removed from that parameter set; see also the
explicit vector-parameter formulation in \cite{Kuksin1989Parameter} and the
subsequent monographs and papers \cite{Kuksin1993,Poschel1996,Kuksin2000}.  For
one-dimensional Schr\"odinger and wave equations, parameters in the linear potential provided the frequency variation required by the KAM scheme.

Wayne's 1990 work on the one-dimensional Dirichlet nonlinear wave equation
provides a probabilistic formulation particularly relevant to the present problem
\cite{Wayne1990}.  Instead of using only a prescribed finite-dimensional
family, Wayne considered an infinite-dimensional class of admissible
$L^2$ potentials and, using inverse spectral coordinates, endowed that
class with a Gaussian probability measure.  The nonresonance conditions
needed for any fixed finite set of modes hold outside a null set of
potentials.  Thus a potential is selected probabilistically and then fixed
before the invariant torus is constructed.  This probabilistic selection of a fixed potential is also central to the quenched formulation of the present paper.  There remains an essential
difference: Wayne's probability measure is introduced through a regular
potential class and its inverse spectral parametrization, so that the
spectral variables themselves are the coordinates used to establish
genericity.  Here the law is prescribed a priori in physical space by
$V=\rho\dot B$, and the potential is a distribution rather than an
$L^2$ function.  The consequences of this distinction for spectral
nondegeneracy are described after the main theorem.

Subsequent developments in the one-dimensional theory concerned equations with a fixed linear part.  Kuksin and P\"oschel treated the cubic Dirichlet NLS
with the fixed constant potential $V(x)\equiv m$, for which the linear
frequencies are $j^2+m$ \cite{KuksinPoschel1996}.  Their Birkhoff normal form yields effective KAM parameters from the small tangential actions and provides the twist required for the Cantor family.  P\"oschel
obtained the analogous fixed-mass result for the one-dimensional Dirichlet
nonlinear wave equation \cite{PoschelWave1996}; a periodic-boundary version
of the one-dimensional NLS was later proved by Geng and You
\cite{GengYou2005}.  Further work considered prescribed spatially varying potentials in place of the constant potential.  Yuan proved the existence of
quasi-periodic solutions for the one-dimensional Dirichlet nonlinear wave
equation with a prescribed smooth nonvanishing potential
\cite{YuanWave2006}, and Du and Yuan constructed finite-dimensional
invariant tori, supported on sufficiently high modes, for the Dirichlet NLS
with a given analytic potential \cite{DuYuan2006}.  In these results, the potential has no adjustable parameters; the KAM parameters are supplied by the tangential actions.

The aforementioned studies concerned equations in one spatial dimension.
In higher dimensions, the unbounded multiplicities of the linear spectrum
and the associated resonant clusters create a different set of difficulties.
We briefly recall two representative approaches.  The first is the
Craig--Wayne--Bourgain (CWB) method.  It originated in the Newton scheme of
Craig and Wayne for nonlinear wave equations \cite{CraigWayne1993} and was
subsequently developed by Bourgain into a multiscale analysis of finite-volume
inverses of the linearized operator on the space--time Fourier lattice
\cite{Bourgain1994}.  This approach yielded periodic solutions for nonlinear
wave equations in higher spatial dimensions \cite{Bourgain1995HigherDim} and
quasi-periodic solutions for Hamiltonian perturbations of the two-dimensional
linear Schr\"odinger equation \cite{Bourgain1998TwoDim}; see also
\cite{Bourgain2005Green}.  An alternative approach retains the iterative structure
of infinite-dimensional KAM theory.  Eliasson and Kuksin introduced the
T\"oplitz--Lipschitz structure to obtain uniform estimates for the block-matrix
homological equations arising from spectral multiplicities and resonant
blocks \cite{EliassonKuksin2010}.  These higher-dimensional developments
address a resonance geometry different from that of the one-dimensional
fixed-potential problem studied here.

Random-potential nonlinear Schr\"odinger dynamics have also been studied in
the lattice setting.  In particular, Bourgain and Wang constructed
quasi-periodic solutions for a lattice NLS with an i.i.d. random potential
\cite{BourgainWang2008}.  Besides, temporal stochastic perturbations provide another setting for KAM
persistence. Zhang, Li, and Wang~\cite{ZhangLiWang2026} study the
persistence of low-dimensional invariant tori of deterministic
nonlinear Schr\"odinger dynamics under stochastic perturbations
in time. Their analysis uses Onsager--Machlup functionals to
characterize the most probable paths and combines large deviation
estimates with KAM theory. The resulting persistence is formulated
in terms of these most probable paths.
In the present paper, the noise instead enters through a
time-independent spatial potential. For almost every realization,
the resulting autonomous equation admits invariant tori carrying
quasi-periodic solutions. After the spatial potential has been fixed, the finite-dimensional KAM
parameters are supplied by the small tangential actions.  Thus the probability-one selection takes place on
Wiener space, whereas the Cantor-set exclusion and its measure estimate are
carried out in action space.  It is in this pathwise sense that the result
is quenched.

To state the main result, let
\[
 J=\{j_1,\dots,j_b\}\subset\N,
 \qquad b=\abs J,
 \qquad
 \Pi_\nu=[\nu,2\nu]^b,
 \qquad \nu>0.
\]
\begin{theorem}[Almost-sure quenched KAM tori]\label{thm:main}
Fix $\rho,\kappa\in\R\setminus\{0\}$.  There is a measurable event $\Omega_*\subset\Omega$ with $\mathbb P(\Omega_*)=1$, independent of $J$, such that the following holds.

For every $\omega\in\Omega_*$, every finite nonempty $J\subset\N$, and every $0<\beta<1/3$, there are constants
\[
 \nu_0=\nu_0(\omega,J,\beta)>0,
 \qquad
 C=C(\omega,J,\beta)>0,
\]
such that, for every $0<\nu<\nu_0$, there is a Cantor set
\[
 \Pi_\nu^*(\omega,J)\subset\Pi_\nu
\]
with
\begin{equation}\label{eq:relative-measure}
 \frac{\abs{\Pi_\nu\setminus\Pi_\nu^*(\omega,J)}}{\abs{\Pi_\nu}}
 \le C\nu^\beta.
\end{equation}
For every $\xi\in\Pi_\nu^*(\omega,J)$, equation~\eqref{eq:NLS} possesses a
real-analytic, linearly stable, $b$-dimensional invariant torus in
$H^1_0(0,\pi;\C)$.  It is a small deformation of the linear rotational torus
supported on $J$ with actions $\xi$.

The corresponding trajectories satisfy
\[
 u\in C(\R;H^1_0(0,\pi))\cap C^1(\R;H^{-1}(0,\pi))
\]
and solve \eqref{eq:NLS} in $H^{-1}$, equivalently in the mild sense.
\end{theorem}

\begin{remark}[Random selection of the potential and the tangential set]
The Brownian path provides an infinite-dimensional external parameter
at the stage of selecting the potential.  Excluding a single event of
probability zero ensures both spectral nonresonance and nonsingularity
of the twist for every finite tangential set,.  After
this selection the path is fixed, and the KAM parameters are the
tangential actions.  Thus the freedom to choose any finite nonempty
$J$ relies on an almost-sure statement about the potential, rather than
a statement for every prescribed potential.  In particular, no
high-mode restriction on $J$ is needed, unlike the prescribed-potential
NLS construction of Du and Yuan~\cite{DuYuan2006}.  The admissible
amplitude threshold still depends on the path and on $J$.
\end{remark}

Two features of the proof deserve emphasis.  The first is qualitative
nondegeneracy in the Brownian variable.  Unlike Wayne's spectral-coordinate
construction, the eigenvalues and eigenfunctions here are correlated
nonlinear functionals of a single Brownian path, so independence among the
spectral data is unavailable.  For every nonzero finitely supported integer
sequence $c=(c_n)$, the map
\[
 B\longmapsto \sum_n c_n\lambda_n(B)
\]
is locally real analytic on Wiener path space and is not identically zero.
An analytic zero-set theorem for Wiener measure therefore implies that its
zero set is null.  Applying the same argument to the twist determinants
yields the simultaneous nondegeneracy described above.  This use of
analyticity supplies no deterministic lower bounds.  Standard
facts about Gaussian and Wiener measures used in this argument may be found
in~\cite{Bogachev1998,Nualart2006}.

The second issue is the normal-form reduction.  A direct attempt to divide
every nontrivial quartic coefficient would require lower bounds for
\[
 \lambda_i+\lambda_j-\lambda_k-\lambda_l
\]
with four unrestricted large indices.  The leading quadratic terms can cancel because of arithmetic
resonances of sums of two squares; for example,
\[
 j^2+(8j)^2=(4j)^2+(7j)^2.
\]
For such quartets, the available asymptotics leave a difference of
spectral errors and provide no lower bound for it.  They therefore do
not suffice for a full quartic Birkhoff normal form.  The KAM construction
used here requires fewer divisors.  P\"oschel's homological equations involve
the Taylor terms of weighted degree at most two, assigning weight two to
action deviations and weight one to normal variables.  Such terms have
normal degree at most two.  We therefore normalize the quartic terms
of normal degree at most two.  Every divisor inverted in this construction contains at most two large normal indices.  The
high-energy growth of the spectrum gives pathwise lower bounds for these
divisors, while weighted matrix estimates show that the resulting generator
is analytic on the energy space.  Terms of normal degree $3$ and $4$ remain
in the perturbation. 
For general background on Birkhoff normal forms for Hamiltonian PDEs, see
\cite{Bambusi2003,BambusiGrebert2006,Grebert2007}.

The rest of this paper is organized as follows.  Section~2 collects the pathwise operator theory, the spectral estimates, and the analytic dependence of the spectral data on the Brownian primitive.  Section~3 proves the
Wiener-space zero-set theorem and derives the almost-sure nonresonance and twist properties.  Section~4 contains the main normal-form argument, including the two-normal divisor and divided-matrix estimates.  Section~5
verifies the hypotheses of the KAM theorem and proves the main result. The scope and the limitations of the argument are collected in the final discussion.

\section{Preliminaries}
We collect the operator-theoretic and spectral results used in the proof of the main theorem.

\subsection{The white-noise operator and its energy space}\label{sec:energy}
For a real $\sigma\in L^2(0,\pi)$ define
\[
 y^{[1]}=y'-\sigma y
\]
and
\[
 \ell_\sigma y=-(y^{[1]})'-\sigma y^{[1]}-\sigma^2y.
\]
The corresponding Dirichlet realization has domain
\[
 \operatorname{Dom}(A_\sigma)
 =\left\{
 y\in L^2:
 y,y^{[1]}\in AC[0,\pi],\ \ell_\sigma y\in L^2,
 \ y(0)=y(\pi)=0
 \right\}.
\]
On this domain, $A_\sigma y=\ell_\sigma y$.  Its closed form has domain
$H^1_0(0,\pi)$ and is given by
\begin{equation}\label{eq:form}
 \mathfrak a_\sigma[u,v]
 =\int_0^\pi u'\overline{v'}\dd x
 -\int_0^\pi \sigma(x)
 \bigl(u'\overline v+u\overline{v'}\bigr)\dd x.
\end{equation}
Formula~\eqref{eq:form} is obtained by interpreting $\sigma'$ distributionally
and integrating once by parts.  It is meaningful for $\sigma\in L^2$ because
$H^1_0(0,\pi)\hookrightarrow L^\infty(0,\pi)$.  The standard theory of
Sturm--Liouville operators with distributional potentials gives
self-adjointness, lower semiboundedness, compact resolvent, simplicity of
the real Dirichlet spectrum, and Sturm oscillation; see
\cite{SavchukShkalikov2003,HrynivMykytyuk2006,HrynivMykytyuk2012}.  When
$\sigma$ is smooth, this realization agrees with the classical operator
$-\partial_x^2+\sigma'$.

For a Brownian path set
\[
 \sigma_\omega=\rho B(\omega),
 \qquad A_\omega=A_{\sigma_\omega}.
\]
Almost surely $B\in C_0([0,\pi])\subset L^2(0,\pi)$.  We enumerate the simple eigenvalues in increasing order,
\[
 \lambda_1(\omega)<\lambda_2(\omega)<\cdots\longrightarrow+\infty,
\]
and choose real normalized eigenfunctions
\[
 A_\omega\phi_n=\lambda_n\phi_n,
 \qquad \norm{\phi_n}_{L^2}=1.
\]
The signs of the real normalized eigenfunctions may be chosen arbitrarily
and are then fixed.  The eigenvalue and twist functionals are independent
of these choices.  The general quartic coefficients also involve odd
powers of individual eigenfunctions; changing $\phi_n$ to $-\phi_n$ is
accompanied by the coordinate change $q_n\mapsto-q_n$, leaving the
physical Hamiltonian unchanged.

The form domain also gives the phase space used below.  Indeed, the
one-dimensional Gagliardo--Nirenberg inequality
\(\norm u_{L^\infty}^2\le \norm u_{L^2}\norm{u'}_{L^2}\)
and Young's inequality give, for every $\eps>0$,
\[
\begin{aligned}
\left|\mathfrak a_\sigma[u,u]-\norm{u'}_{L^2}^2\right|
&\le
2\|\sigma\|_{L^2}
\|u'\|_{L^2}
\|u\|_{L^\infty}\\
&\le
2\|\sigma\|_{L^2}
\|u'\|_{L^2}^{3/2}
\|u\|_{L^2}^{1/2}\\
&\le
\eps\|u'\|_{L^2}^2
+C_{\eps}\|\sigma\|_{L^2}^4
\|u\|_{L^2}^2,
\end{aligned}
\]
Here the form perturbation is $-2\operatorname{Re}\int_0^\pi
\sigma u'\overline u\dd x$.  For fixed real $\sigma\in L^2$, choosing
$\eps=1/2$ gives a constant $K_\sigma\ge0$ such that
\[
 \tfrac12\norm{u'}_{L^2}^2-K_\sigma\norm u_{L^2}^2
 \le\mathfrak a_\sigma[u,u]
 \le\tfrac32\norm{u'}_{L^2}^2+K_\sigma\norm u_{L^2}^2.
\]
Since $\norm u_{L^2}\le\norm{u'}_{L^2}$ on $H^1_0(0,\pi)$,
any $\mu_\sigma\ge\max(2,K_\sigma+3/2)$ then satisfies
\begin{equation}\label{eq:form-equivalence}
 \mu_\sigma^{-1}\norm u_{H^1}^2
 \le \mathfrak a_\sigma[u,u]+\mu_\sigma\norm u_{L^2}^2
 \le \mu_\sigma\norm u_{H^1}^2.
\end{equation}
In particular, $A_\sigma+\mu_\sigma$ is strictly positive, and
$\operatorname{Dom}((A_\sigma+\mu_\sigma)^{1/2})=H^1_0(0,\pi)$,
with equivalent norms.
If $u=\sum q_n\phi_n$, the spectral theorem for the closed form gives
\begin{equation}\label{eq:spectral-form-norm}
 \norm u_{H^1}^2\asymp_\omega
 \sum_{n\ge1}(\mu_\omega+\lambda_n)\abs{q_n}^2.
\end{equation}
The high-energy asymptotics below imply equivalence with
\[
 \norm q_1^2:=\sum_{n\ge1}\langle n\rangle^2\abs{q_n}^2,
 \qquad \langle n\rangle=(1+n^2)^{1/2}.
\]
We write
\[
 h^1=\left\{q:\norm q_1<\infty\right\}
\]
and identify it with the pathwise energy space $h^1_\omega\simeq H^1_0(0,\pi)$.  Equivalence constants may depend on the path.

\subsection{High-energy spectral estimates}
The following are the probabilistic high-energy estimates used
in the normal-form and KAM arguments.  They follow from the stochastic
Pr\"ufer analysis in~\cite{DongJianYuan2026}.

\begin{proposition}[Pathwise spectral estimates]\label{prop:spectral-estimates}
There is an event $\Omega_{\mathrm{sp}}$ of probability one such that, for every $\omega\in\Omega_{\mathrm{sp}}$ and every $\eps>0$,
\begin{equation}\label{eq:eig-asymptotic}
 \lambda_n(\omega)=n^2+O_{\omega,\eps}(n^\eps),
 \qquad n\to\infty,
\end{equation}
and
\begin{equation}\label{eq:eigfun-uniform}
 \sup_{n\ge1}\norm{\phi_n(\omega)}_{L^\infty(0,\pi)}<\infty.
\end{equation}
In particular, for every fixed $0<\eps<1$,
\begin{equation}\label{eq:gap-asymptotic}
 \lambda_{n+1}-\lambda_n=2n+1+O_{\omega,\eps}(n^\eps),
\end{equation}
and there is $N_\omega$ such that
\begin{equation}\label{eq:rough-growth}
 \tfrac12 n^2\le\lambda_n\le2n^2,
 \qquad
 \lambda_{n+1}-\lambda_n\ge n,
 \qquad n\ge N_\omega.
\end{equation}
\end{proposition}

\begin{remark}
The result cited in \cref{prop:spectral-estimates} is stronger: on an interval of length $L$ it gives a stochastic first correction to $\sqrt{\lambda_n}$ and a uniform first-order eigenfunction expansion.  For the present KAM argument only \eqref{eq:eig-asymptotic} and \eqref{eq:eigfun-uniform} are used.  Since the almost-sure estimate is initially stated separately for each $\eps>0$, intersecting the full-probability events for positive rational $\eps$ produces the single event $\Omega_{\mathrm{sp}}$ used here.
\end{remark}

\subsection{Analyticity with respect to the primitive}
 We establish the analytic dependence of the spectral data on the Brownian primitive. We use the standard Banach-space meaning of analyticity: near every point the function is represented by a norm-convergent series of continuous homogeneous polynomials. Let
\[
 E=C_0([0,\pi];\R)
 =\{B\in C([0,\pi];\R):B(0)=0\}
\]
with the uniform norm, and write $A(B)=A_{\rho B}$.

\begin{proposition}[Local analyticity of eigenpairs]\label{prop:eigen-analytic}
For every $n\ge1$, the map $B\mapsto\lambda_n(B)$ is continuous and locally real analytic on $E$.  Near every $B_0\in E$ one can choose a locally analytic branch
\[
 B\longmapsto\phi_n(B)\in H^1_0(0,\pi)
\]
normalized by
\[
 \int_0^\pi \phi_n(B,x)^2\dd x=1.
\]
Consequently, for fixed $m,n$, the function
\[
 B\longmapsto\int_0^\pi \phi_m(B,x)^2\phi_n(B,x)^2\dd x
\]
and every finite polynomial expression formed from such even eigenfunction functionals are globally well defined and locally real analytic.
\end{proposition}

\begin{proof}
We first allow the primitive to be complex in order to establish
holomorphic dependence, and then restrict to real primitives.  Write
$X=C([0,\pi];\C)$, with the uniform norm.  The primitive associated with $B$ is $\sigma=\rho B$.

\medskip
\noindent\textit{Step 1: the initial-value problem and its analytic dependence.}
For $(\sigma,\lambda)\in X\times\C$, set
\[
 M_{\sigma,\lambda}(x)=
 \begin{pmatrix}
  \sigma(x)&1\\
  -(\lambda+\sigma(x)^2)&-\sigma(x)
 \end{pmatrix}.
\]
The relations $q=y^{[1]}=y'-\sigma y$ and
$\ell_\sigma y=\lambda y$ are equivalent to
\begin{equation}\label{eq:fundamental-system}
 \binom{y}{q}'=M_{\sigma,\lambda}\binom{y}{q},
 \qquad (y(0),q(0))=(0,1).
\end{equation}
The prime denotes differentiation with respect to $x$.

Let $Z_*=(0,1)^{\mathsf T}$ and define recursively
\[
 Z_0(x)=Z_*,\qquad
 Z_{k+1}(x)=\int_0^x M_{\sigma,\lambda}(s)Z_k(s)\dd s,
 \qquad k\ge0.
\]
For every $R>0$, there is $K_R<\infty$ such that
$\sup_x\norm{M_{\sigma,\lambda}(x)}\le K_R$ whenever
$\norm\sigma_\infty\le R$ and $|\lambda|\le R$.
Induction yields
\[
 |Z_k(x)|\le |Z_*|\frac{(K_Rx)^k}{k!},\qquad 0\le x\le\pi.
\]
Consequently the series
\[
 Z(x;\sigma,\lambda)=\sum_{k=0}^\infty Z_k(x)
\]
converges in $C([0,\pi];\C^2)$, uniformly on every bounded parameter
set. Summing the recursion yields the integral equation
\[
 Z(x)=Z_*+\int_0^x M_{\sigma,\lambda}(s)Z(s)\dd s.
\]
The resulting function is the unique solution of \eqref{eq:fundamental-system}.  

The map $(\sigma,\lambda)\mapsto M_{\sigma,\lambda}$ is a continuous polynomial with values in the space of continuous matrix-valued functions.  Multiplication
and integration are bounded multilinear or linear operations on these
spaces, so each $Z_k$ is a continuous polynomial in $(\sigma,\lambda)$.
By the Weierstrass convergence theorem for Banach-space-valued holomorphic maps~\cite{Complexanalysisinfidim}, the locally uniform convergence implies that $Z$ is jointly holomorphic as a $C([0,\pi];\C^2)$-valued map.   This conclusion also
holds with values in $C^1([0,\pi];\C^2)$: the integration map
\[
 \mathcal I:C([0,\pi];\C^2)\longrightarrow C^1([0,\pi];\C^2),
 \qquad (\mathcal I f)(x)=\int_0^x f(s)\dd s,
\]
is bounded, and the identity $Z=Z_*+\mathcal I(M_{\sigma,\lambda}Z)$
expresses $Z$ as a holomorphic map into $C^1$.
In particular, the endpoint function
\[
 D(\sigma,\lambda)=y(\pi;\sigma,\lambda)
\]
is jointly holomorphic.

\medskip
\noindent\textit{Step 2: eigenvalues are simple zeros of the endpoint function.}
Fix a real primitive $\sigma_0\in C([0,\pi];\R)$.
If $D(\sigma_0,\lambda)=0$, the initial solution $y$ is nonzero,
satisfies both Dirichlet conditions, and belongs to the operator domain:
$y,q\in C^1\subset AC$ and $\ell_{\sigma_0}y=\lambda y\in L^2$.
Hence $\lambda$ is a Dirichlet eigenvalue.  Conversely, an eigenfunction
and its quasi-derivative solve the same first-order system.  Its initial
quasi-derivative cannot vanish, since together with $y(0)=0$ this would
force the entire solution to be zero by uniqueness.  Rescaling that
initial quasi-derivative to $1$ therefore gives the initial solution
above, and its endpoint is zero.

To justify the nonvanishing derivative needed for the implicit-function
theorem, fix $\lambda_0=\lambda_n(\sigma_0)$ and abbreviate
\[
 y=y(\cdot;\sigma_0,\lambda_0),\quad
 q=q(\cdot;\sigma_0,\lambda_0),\quad
 z=\partial_\lambda y(\cdot;\sigma_0,\lambda_0),\quad
 r=\partial_\lambda q(\cdot;\sigma_0,\lambda_0).
\]
All four functions are $C^1$, and differentiation of
\eqref{eq:fundamental-system} gives
\[
 z'=\sigma_0z+r,\qquad
 r'=-(\lambda_0+\sigma_0^2)z-\sigma_0r-y,
 \qquad z(0)=r(0)=0.
\]
The quasi-derivative Wronskian $W=zq-ry$ consequently satisfies
\[
 \begin{aligned}
 W'&=(\sigma_0z+r)q
      +z\bigl(- (\lambda_0+\sigma_0^2)y-\sigma_0q\bigr)\\
    &\quad-\bigl(- (\lambda_0+\sigma_0^2)z-\sigma_0r-y\bigr)y
      -r(\sigma_0y+q)
     =y^2.
 \end{aligned}
\]
Since $W(0)=0$ and $y(\pi)=0$, integration yields
\begin{equation}\label{eq:endpoint-lambda-derivative}
 q(\pi)\,\partial_\lambda D(\sigma_0,\lambda_0)
 =\int_0^\pi y(x)^2\dd x>0.
\end{equation}
The integral is positive because $y$ is real and nonzero.
Also $q(\pi)\ne0$, since otherwise terminal data
$(y(\pi),q(\pi))=(0,0)$ would force $y=q=0$.
In particular,
$\partial_\lambda D(\sigma_0,\lambda_0)\ne0$.
Thus the zero of $D$ is simple.

\medskip
\noindent\textit{Step 3: continuity of the ordered eigenvalues.}
From this step onward, we restrict the primitive to the real Banach space
\[
 X_{\R}=C([0,\pi];\R)\subset X
\]
and the spectral parameter to $\R$.  For real parameters, the coefficient
matrix and initial data in \eqref{eq:fundamental-system} are real, so
uniqueness implies that $y$ and $q$ are real-valued.  In particular,
$D:X_{\R}\times\R\to\R$ is real analytic by restriction of the
holomorphic map in Step~1.  Its Taylor coefficients in real directions
are real, since they are limits of real difference quotients.
We prove that the ordered eigenvalue \(\lambda_n(\sigma)\) depends continuously on \(\sigma\). This will allow us to identify it with the local analytic branch obtained from the implicit-function theorem near \((\sigma_0,\lambda_n(\sigma_0))\).

Choose $\mu>0$ and $c>0$ such that the positive shifted form
\[
 \mathfrak b_0[u,u]
 :=\mathfrak a_{\sigma_0}[u,u]+\mu\norm u_{L^2}^2
 \ge c\norm u_{H^1}^2,
 \qquad u\in H^1_0(0,\pi).
\]
Such constants exist by the form estimate preceding
\eqref{eq:form-equivalence}.  For another real continuous primitive
$\sigma$, put $\delta=\norm{\sigma-\sigma_0}_\infty$ and
$\mathfrak b_\sigma=\mathfrak a_\sigma+\mu\ip{\cdot}{\cdot}_{L^2}$.
Formula~\eqref{eq:form} and Young's inequality give
\[
 \begin{aligned}
 |\mathfrak b_\sigma[u,u]-\mathfrak b_0[u,u]|
 &\le 2\delta\norm{u'}_{L^2}\norm u_{L^2}\\
 &\le\delta\norm u_{H^1}^2
 \le\frac\delta c\,\mathfrak b_0[u,u].
 \end{aligned}
\]
If $\eta=\delta/c<1$, it follows that
\[
 (1-\eta)\mathfrak b_0[u,u]
 \le\mathfrak b_\sigma[u,u]
 \le(1+\eta)\mathfrak b_0[u,u].
\]
All these forms have the common domain $H^1_0$, whose embedding into
$L^2$ is compact.  Apply the min--max formula
\[
 \lambda_n(\sigma)+\mu
 =\inf_{\substack{V\subset H^1_0\\\dim V=n}}
   \ \sup_{0\ne u\in V}
   \frac{\mathfrak b_\sigma[u,u]}{\norm u_{L^2}^2}
\]
to the preceding inequalities.  We obtain
\[
 (1-\eta)(\lambda_n(\sigma_0)+\mu)
 \le\lambda_n(\sigma)+\mu
 \le(1+\eta)(\lambda_n(\sigma_0)+\mu).
\]
Thus $\lambda_n(\sigma)\to\lambda_n(\sigma_0)$ as
$\norm{\sigma-\sigma_0}_\infty\to0$.  This argument includes negative
eigenvalues, since the positive shift $\mu\norm u_{L^2}^2$ is fixed throughout the comparison.

\medskip
\noindent\textit{Step 4: the analytic eigenvalue and eigenfunction branches.}
By Step~2, the real analytic implicit-function theorem in Banach spaces
applies to $D:X_{\R}\times\R\to\R$ at $(\sigma_0,\lambda_0)$.
There are an open neighborhood $U\subset X_{\R}$ of $\sigma_0$,
an open interval $V\subset\R$ containing $\lambda_0$,
and a real analytic function $\Lambda:U\to V$ such that
\[
 \Lambda(\sigma_0)=\lambda_0,\qquad
 D(\sigma,\Lambda(\sigma))=0,
\]
and every zero of $D$ in $U\times V$ lies on this graph.
For real $\sigma$ close to $\sigma_0$, the ordered eigenvalue $\lambda_n(\sigma)$ belongs to $V$ by Step~3 and is a zero of $D$ by Step~2.  Uniqueness of the graph therefore gives
$\Lambda(\sigma)=\lambda_n(\sigma)$.
Thus $\lambda_n$ is real analytic on $U$, after shrinking $U$ if necessary.

Now put
\[
 Y(\sigma)=y(\cdot;\sigma,\Lambda(\sigma)).
\]
By Step~1 and composition with $\Lambda$, the map $Y$ is real analytic
with values in $C^1([0,\pi];\C)$, hence in $H^1(0,\pi;\C)$,
regarded here as real Banach spaces.  For every $\sigma\in U$, the
initial-value solution $Y(\sigma)$ is real-valued by Step~3.  The initial
condition gives vanishing at the left endpoint, while the equation
$D(\sigma,\Lambda(\sigma))=0$ gives vanishing at the right endpoint.
Thus $Y(\sigma)\in H^1_0(0,\pi;\R)$ for every $\sigma\in U$.
Since $H^1_0(0,\pi;\R)$ is closed in $H^1(0,\pi;\C)$, real
difference quotients of $Y$ converge within this subspace.  Applying the
same argument successively to its derivatives gives
\[
 D^mY(\sigma)[h_1,\ldots,h_m]\in H^1_0(0,\pi;\R),
 \qquad h_j\in X_{\R},\quad m\ge1.
\]
The inherited norm gives the same Fr\'echet remainder estimates and
convergence of the Taylor series.  Its coefficients
$P_m(h)=D^mY(\sigma)[h,\ldots,h]/m!$ are therefore continuous
homogeneous polynomials with values in $H^1_0(0,\pi;\R)$, proving
real analyticity into this space.

Define
\[
 S(\sigma)=\int_0^\pi Y(\sigma,x)^2\dd x.
\]
This is real analytic, being the composition of $Y$ with a continuous
quadratic polynomial on $H^1_0(0,\pi;\R)$.  Since $S(\sigma_0)>0$,
we may shrink $U$ so that $S>0$ and define
\[
 \Phi_n(\sigma)=Y(\sigma)/\sqrt{S(\sigma)}.
\]
The positive square root is real analytic on $(0,\infty)$, so $\Phi_n$
is a real analytic branch of real normalized eigenfunctions.
Composition with $\sigma=\rho B$ proves the assertion on $E$; see
also the perturbation framework in~\cite{Kato1995}.

\medskip
\noindent\textit{Step 5: even eigenfunction functionals.}
The map
\[
 (f_1,f_2,f_3,f_4)\longmapsto
 \int_0^\pi f_1f_2f_3f_4\dd x
\]
is a continuous four-linear form on $(H^1_0)^4$.  Indeed, the
one-dimensional Sobolev embedding gives
\[
 \left|\int_0^\pi f_1f_2f_3f_4\dd x\right|
 \le\norm{f_1}_{L^\infty}\norm{f_2}_{L^\infty}
      \norm{f_3}_{L^2}\norm{f_4}_{L^2}
 \le C\prod_{j=1}^4\norm{f_j}_{H^1}.
\]
Using the local analytic branches for $m$ and $n$ on a common
neighborhood therefore proves local real analyticity of
$\int\phi_m^2\phi_n^2$.
At each real primitive, any two real normalized eigenfunctions for the
same eigenvalue differ by $+1$ or $-1$, because that eigenvalue is
simple.  Their squares coincide, so these local expressions agree on
overlaps and define a single function on all of $E$.
Finite sums and products of such functions are again locally real
analytic and independent of the choices of signs.  This proves the final
assertion.
\end{proof}

We shall also need the derivatives of the eigenvalues in
Cameron--Martin directions.  The Cameron--Martin space of Wiener measure
on $E$ is
\[
 H_W=\left\{H_g(x)=\int_0^x g(s)\dd s:\ g\in L^2(0,\pi)\right\},
 \qquad \norm{H_g}_{H_W}=\norm g_{L^2}.
\]
Note that a Cameron--Martin path vanishes at $0$, but need not vanish at $\pi$.

\begin{proposition}[Feynman--Hellmann formula]\label{prop:FH}
For $g\in L^2(0,\pi)$ and $H_g(x)=\int_0^xg(s)\dd s$,
\begin{equation}\label{eq:FH}
 D_{H_g}\lambda_n(B)
 =\rho\int_0^\pi g(x)\phi_n(B,x)^2\dd x.
\end{equation}
\end{proposition}

\begin{proof}
Fix $B\in E$ and $g\in L^2(0,\pi;\R)$.  Since
$\norm{H_g}_\infty\le\sqrt\pi\norm g_{L^2}$, the line $B+tH_g$
lies in $E$.  On the common real form domain $H^1_0(0,\pi;\R)$, write
$\mathfrak a_t=\mathfrak a_{\rho(B+tH_g)}$.  Then
\[
 \mathfrak a_t=\mathfrak a_0+t\mathfrak v_g,
 \qquad
 |\mathfrak v_g[u,v]|
 \le C|\rho|\norm g_{L^2}\norm u_{H^1}\norm v_{H^1}.
\]
Integration by parts gives
\begin{equation}\label{eq:FH-form-variation}
 \mathfrak v_g[u,v]
 =-\rho\int_0^\pi H_g(uv)'\dd x
 =\rho\int_0^\pi guv\dd x.
\end{equation}
The boundary term vanishes because $uv$ has zero trace; no condition
on $H_g(\pi)$ is needed.

Choose the real analytic normalized eigenpair
$\lambda(t)=\lambda_n(B+tH_g)$, $\phi(t)=\phi_n(B+tH_g)$ supplied
by \cref{prop:eigen-analytic}.  Its weak eigenvalue identity is
\begin{equation}\label{eq:FH-weak-eigenvalue}
 \mathfrak a_t[\phi(t),v]=\lambda(t)\ip{\phi(t)}v_{L^2},
 \qquad v\in H^1_0(0,\pi;\R).
\end{equation}
Set $\phi=\phi(0)$ and $\dot\phi=\phi'(0)\in H^1_0$.
Differentiating $\lambda(t)=\mathfrak a_t[\phi(t),\phi(t)]$ gives
\begin{equation}\label{eq:FH-energy-derivative}
 \lambda'(0)=\mathfrak v_g[\phi,\phi]
             +2\mathfrak a_0[\phi,\dot\phi]
 =\mathfrak v_g[\phi,\phi]
             +2\lambda(0)\ip\phi{\dot\phi}_{L^2}.
\end{equation}
The last term vanishes by differentiation of
$\norm{\phi(t)}_{L^2}^2=1$, proving \eqref{eq:FH}.  The argument uses
only the common form domain.  Finally,
\[
 |D_{H_g}\lambda_n(B)|
 \le |\rho|\norm{\phi_n(B)^2}_{L^2}\norm g_{L^2},
\]
so the derivative is continuous in the Cameron--Martin norm.
\end{proof}

\section{Analytic nondegeneracy on Wiener space}
In this section, we first establish an analytic zero set theorem on classical Wiener space.  The proof is based on the density of the Cameron--Martin space, a one-dimensional Gaussian decomposition, and the
identity theorem for analytic functions. Then we apply this result to obtain nonresonance and twist properties.

\begin{theorem}[Analytic zero sets on classical Wiener space]\label{thm:wiener-zero}
Let $E=C_0([0,\pi];\R)$ carry classical Wiener measure $\mu$.  If $f:E\to\R$ is continuous, locally real analytic, and not identically zero, then
\[
 \mu\{B\in E:f(B)=0\}=0.
\]
\end{theorem}

\begin{proof}
\medskip
\noindent\textit{Step 1: a countable dense collection of directions.}
For each $j\ge1$, consider paths that are linear on the intervals of the
partition $\{k\pi 2^{-j}:0\le k\le2^j\}$, vanish at $0$, and have rational
values at the remaining partition points.  Their union, denoted by
$H_0$, is countable and is a vector space over $\mathbb Q$.  Every such
path is absolutely continuous with piecewise constant derivative in
$L^2$, so $H_0\subset H_W$.  Moreover, uniform continuity and linear
interpolation show that every path in $E$ can be approximated uniformly
by these paths.  Thus $H_0$ is dense in $E$.

Enumerate its nonzero elements as $h_1,h_2,\ldots$, and write
$h_m=H_{g_m}$.  Define
\[
 Z_m=\{B\in E:f(B)=0,\quad
             t\mapsto f(B+th_m)\text{ is not identically zero}\}.
\]
This is a Borel set, because continuity in $t$ gives the explicit formula
\[
 Z_m=f^{-1}(0)\cap
       \bigcup_{q\in\mathbb Q}\{B:f(B+qh_m)\ne0\}.
\]

\medskip
\noindent\textit{Step 2: decomposition along one Gaussian direction.}
Fix $m$, put $a_m=\norm{g_m}_{L^2}>0$, and set
$\widehat g_m=g_m/a_m$ and $\widehat h_m=h_m/a_m$.
On the canonical Wiener probability space define
\[
 Z=\int_0^\pi\widehat g_m(s)\dd B_s,
 \qquad B_x^\perp=B_x-Z\widehat h_m(x).
\]
The Wiener integral with deterministic integrand is a centered Gaussian variable with variance $\norm{\widehat g_m}_{L^2}^2=1$.
Also $B^\perp$ has continuous paths vanishing at $0$, and hence is an
$E$-valued random element. Fix spatial points $x_1,\ldots,x_r\in[0,\pi]$.
By the construction of $h_m$, its normalized derivative
$\widehat g_m$ is a step function.  Thus, for a suitable partition
$0=t_0<t_1<\cdots<t_N=\pi$ and real coefficients $a_j$,
\[
 \widehat g_m
 =\sum_{j=1}^N a_j\mathbf1_{(t_{j-1},t_j]},
 \qquad
 Z=\sum_{j=1}^N a_j(B_{t_j}-B_{t_{j-1}}).
\]
The vector consisting of the Brownian values at all the points
$x_1,\ldots,x_r,t_0,\ldots,t_N$ is centered Gaussian.
Consequently,
\[
 (B_{x_1},\ldots,B_{x_r},Z)
\]
is centered Gaussian, since it is a linear image of that vector.
The same is therefore true of
\[
 (B^\perp_{x_1},\ldots,B^\perp_{x_r},Z),
 \qquad
 B^\perp_x=B_x-Z\widehat h_m(x).
\]

Since $B_x=\int_0^\pi\mathbf1_{[0,x]}(s)\dd B_s$, the covariance
identity for deterministic Wiener integrals gives
\[
 \mathbb E[B_xZ]=\int_0^x\widehat g_m(s)\dd s=\widehat h_m(x),
 \qquad \mathbb E[Z^2]=1.
\]
Consequently,
\[
 \mathbb E[B_x^\perp Z]
 =\mathbb E[B_xZ]-\widehat h_m(x)\mathbb E[Z^2]=0.
\]

Set
$Y=(B^\perp_{x_1},\ldots,B^\perp_{x_r})$ and let
$\Sigma$ be its covariance matrix.
The centered Gaussian vector $(Y,Z)$ has covariance matrix
\[
 \begin{pmatrix}
  \Sigma&0\\
  0&1
 \end{pmatrix}.
\]
Therefore, $Y$ and $Z$ are
independent. Evaluations on a countable dense subset of $[0,\pi]$ generate the
Borel sigma-algebra of $E$, so the entire random element $B^\perp$ is
independent of $Z$.  Denote its law by $\mu_m^\perp$ and the standard
one-dimensional Gaussian law by $\gamma_1$.  We have the representation
\[
 B=B^\perp+Z\widehat h_m,
 \qquad \operatorname{Law}(B^\perp,Z)=\mu_m^\perp\otimes\gamma_1.
\]

\medskip
\noindent\textit{Step 3: the one-dimensional zero-set argument.}
For a fixed $y\in E$, the function
$F_y(z)=f(y+z\widehat h_m)$ is locally real analytic on $\R$, by
restriction of the local power series for $f$ to this affine line.
If $F_y$ is not identically zero, the one-dimensional identity theorem
implies that its zeros are isolated.  An isolated subset of $\R$ is
countable, so its Gaussian measure is zero.  If $F_y$ is identically
zero, no point on this line belongs to $Z_m$: translating or rescaling
the parameter does not change whether the restriction is identically
zero.  Consequently, in both cases,
\[
 \int_\R\mathbf1_{Z_m}(y+z\widehat h_m)\dd\gamma_1(z)=0.
\]
The integrand is measurable by Step~1.  Independence and Fubini's theorem
now give
\[
 \mu(Z_m)=\int_E\int_\R
   \mathbf1_{Z_m}(y+z\widehat h_m)
       \dd\gamma_1(z)\dd\mu_m^\perp(y)=0.
\]

\medskip
\noindent\textit{Step 4: covering the zero set.}
Suppose that $f(B)=0$ and $B\notin\bigcup_m Z_m$.  By the definition of
$Z_m$, the restriction $t\mapsto f(B+th_m)$ must then vanish identically
for every $m$.  In particular, $f(B+h)=0$ for every $h\in H_0$, including
$h=0$.  Given any $Y\in E$, choose $h^{(j)}\in H_0$ with
$h^{(j)}\to Y-B$ uniformly.  Continuity gives
\[
 f(Y)=\lim_{j\to\infty}f(B+h^{(j)})=0.
\]
This contradicts the assumption that $f$ is not identically zero.
Therefore $f^{-1}(0)\subset\bigcup_m Z_m$.  The union is countable and
each member has measure zero, which proves the theorem.
\end{proof}

\subsection{Almost-sure spectral nonresonance and twist}

\paragraph{Integer spectral relations.}
For a finitely supported real sequence $c=(c_n)$ put
\[
 F_c(B)=\sum_n c_n\lambda_n(B).
\]

\begin{proposition}\label{prop:combination-nontrivial}
If $c\ne0$ has finite support, then $F_c$ is not identically zero on $E$.
\end{proposition}

\begin{proof}
Let $S=\supp c$, a finite nonempty set.  By
\cref{prop:eigen-analytic}, the finite sum $F_c$ is locally real analytic.
At the free path $B=0$,
\[
 \lambda_n(0)=n^2,\qquad
 e_n(x)=\sqrt{\frac2\pi}\sin(nx),\qquad
 e_n(x)^2=\frac1\pi(1-\cos(2nx)).
\]
Define the smooth real function
\[
 g_c(x)=\sum_{n\in S}c_ne_n(x)^2
 =\frac1\pi\sum_{n\in S}c_n
   -\frac1\pi\sum_{n\in S}c_n\cos(2nx).
\]
Orthogonality of the constant and cosine functions gives
\begin{equation}\label{eq:gc-norm}
 \norm{g_c}_{L^2}^2
 =\frac1\pi\left(\sum_nc_n\right)^2
   +\frac1{2\pi}\sum_nc_n^2>0.
\end{equation}
The path $H_{g_c}(x)=\int_0^xg_c(s)\dd s$ belongs to $H_W$.
Differentiating the finite sum and applying \cref{prop:FH} yields
\[
 \begin{aligned}
 D_{H_{g_c}}F_c(0)
 &=\rho\sum_{n\in S}c_n\int_0^\pi g_c(x)e_n(x)^2\dd x\\
 &=\rho\int_0^\pi g_c(x)
                  \left(\sum_{n\in S}c_ne_n(x)^2\right)\dd x
 =\rho\norm{g_c}_{L^2}^2\ne0.
 \end{aligned}
\]
Thus $F_c$ is not identically zero.
\end{proof}

\begin{theorem}[Universal finite-support nonresonance]\label{thm:universal-nr}
There is an event $\Omega_{\mathrm{nr}}$ of probability one such that, for every $\omega\in\Omega_{\mathrm{nr}}$,
\begin{equation}\label{eq:all-integer-nr}
 \sum_n c_n\lambda_n(\omega)\ne0
\end{equation}
for every nonzero finitely supported integer sequence $c\in\Z^{(\N)}$.
\end{theorem}

\begin{proof}
Fix a nonzero finitely supported integer sequence $c$.  The map
$F_c:E\to\R$ is continuous and locally real analytic by
\cref{prop:eigen-analytic}, and is not identically zero by
\cref{prop:combination-nontrivial}.  The preceding zero-set theorem gives
\[
 \mu\{B:F_c(B)=0\}=0.
\]
Since $\Z^{(\N)}=\bigcup_{N\ge1}\Z^N$ under the natural embeddings,
this set of labels is countable.  Thus
\[
 E_{\mathrm{nr}}
 =\bigcap_{0\ne c\in\Z^{(\N)}}\{B\in E:F_c(B)\ne0\}
\]
is a Borel set of Wiener measure one.  Its inverse image
$\Omega_{\mathrm{nr}}=\{\omega:B(\omega)\in E_{\mathrm{nr}}\}$
has probability one and is independent of $J$.
\end{proof}

\paragraph{The quartic twist.}
For $m,n\ge1$ define
\begin{equation}\label{eq:Cmn}
 C_{mn}(B)=\int_0^\pi\phi_m(B,x)^2\phi_n(B,x)^2\dd x.
\end{equation}
For $J=\{j_1,\dots,j_b\}$ define the symmetric matrix $A_J(B)$ by
\begin{equation}\label{eq:AJ}
 (A_J)_{aa}=\kappa C_{j_aj_a},
 \qquad
 (A_J)_{ac}=2\kappa C_{j_aj_c}\quad(a\ne c).
\end{equation}
This matrix is the Hessian of the tangential quartic action polynomial
$Z_{TT}$ defined in \eqref{eq:ZTT}.

\begin{theorem}[Almost-sure twist for every finite set]\label{thm:twist}
There is an event $\Omega_{\mathrm{tw}}$ of probability one such that, for every $\omega\in\Omega_{\mathrm{tw}}$ and every finite nonempty $J\subset\N$,
\[
 \det A_J(\omega)\ne0.
\]
\end{theorem}

\begin{proof}
Fix $J=\{j_1,\ldots,j_b\}$.  By \cref{prop:eigen-analytic},
$B\mapsto\det A_J(B)$ is globally well defined, continuous, and locally
real analytic.  To verify that it is not identically zero, evaluate it
at the free path.  The normalized eigenfunctions
$e_n(x)=\sqrt{2/\pi}\sin(nx)$ give
\[
 C_{nn}(0)=\frac3{2\pi},\qquad
 C_{mn}(0)=\frac1\pi\quad(m\ne n).
\]
Let $\bm1=(1,\ldots,1)^{\mathsf T}\in\R^b$ and let $I_b$ be the
identity matrix.  Then
\begin{equation}\label{eq:AJ-free}
 A_J(0)=\frac{2\kappa}{\pi}
 \left(\bm1\bm1^{\mathsf T}-\frac14I_b\right).
\end{equation}
The matrix $\bm1\bm1^{\mathsf T}$ acts as multiplication by $b$ on
$\Span\{\bm1\}$ and vanishes on its orthogonal complement.  Hence
\[
 \det A_J(0)
 =\frac{2\kappa}{\pi}\left(b-\frac14\right)
       \left(-\frac\kappa{2\pi}\right)^{b-1}\ne0.
\]

By \cref{thm:wiener-zero}, the zero set of $\det A_J$ has Wiener
measure zero.  Intersecting its complement over the countable family
of finite nonempty subsets $J\subset\N$, and taking the inverse image
under $\omega\mapsto B(\omega)$, gives $\Omega_{\mathrm{tw}}$.
\end{proof}

\section{Partial quartic Birkhoff normal form}\label{sec:BNF}

\subsection{Hamiltonian formulation and sign conventions}\label{sec:Hamiltonian}
Fix henceforth
\[
 \omega\in\Omega_{\mathrm{sp}}\cap\Omega_{\mathrm{nr}}\cap\Omega_{\mathrm{tw}}
\]
and suppress the path from the notation.  Write
\[
 u=\sum_{n\ge1}q_n\phi_n,
 \qquad
 v=\sum_{n\ge1}\bar q_n\phi_n,
\]
Here $q$ and $\bar q$ are independent complex coordinates; the original
NLS phase space is recovered by imposing $\bar q=\overline q$, so that
$v=\overline u$.

On $h^1_\C\times h^1_\C$ use the weak complex symplectic form
\begin{equation}\label{eq:symp}
 \Om_{\mathrm{symp}}=i\sum_{n\ge1}\dd q_n\wedge\dd\bar q_n.
\end{equation}
We fix the convention
\begin{equation}\label{eq:Hamiltonian-convention}
 \iota_{X_F}\Om_{\mathrm{symp}}=\dd F.
\end{equation}
Thus
\begin{equation}\label{eq:XF}
 X_F=(-i\partial_{\bar q}F,\ i\partial_qF)
\end{equation}
and
\begin{equation}\label{eq:Poisson}
 \{F,G\}=\dd F[X_G]
 =-i\sum_{n\ge1}
 \left(F_{q_n}G_{\bar q_n}-F_{\bar q_n}G_{q_n}\right).
\end{equation}
With this convention, $i\dot q_n=\partial_{\bar q_n}H$.

The Hamiltonian is
\begin{equation}\label{eq:Hamiltonian}
 H=H_2+P_4,
 \qquad
 H_2=\sum_{n\ge1}\lambda_nq_n\bar q_n,
\end{equation}
where
\begin{equation}\label{eq:P4}
 P_4(q,\bar q)
 =\frac\kappa2\int_0^\pi u(x)^2v(x)^2\dd x.
\end{equation}

The quadratic functional $H_2$ is continuous on $h^1\times h^1$, because
$\lambda_n=O_\omega(n^2)$.  Its formal Hamiltonian vector field, however,
is unbounded on $h^1$.  Whenever a Poisson bracket with $H_2$ appears
below, it is therefore understood as the continuous dual pairing
\[
 \{H_2,F\}=\dd H_2[X_F],
\]
where $X_F$ takes values in $h^1\times h^1$.  The conjugacy to the
original equation will also be formulated in the weak symplectic sense.

\begin{proposition}[Analyticity of the local cubic vector field]\label{prop:P4-analytic}
The vector field $X_{P_4}$ is a real analytic cubic map
\[
 X_{P_4}:h^1\times h^1\longrightarrow h^1\times h^1
\]
and
\begin{equation}\label{eq:P4-cubic}
 \norm{X_{P_4}(q,\bar q)}_{h^1\times h^1}
 \le C_\omega(\norm q_1+\norm{\bar q}_1)^3.
\end{equation}
\end{proposition}

\begin{proof}
Denote the spectral transform by
$\mathcal Uq=\sum_nq_n\phi_n$.  The form-norm equivalence and the
spectral asymptotics give bounded inverse maps
\[
 \mathcal U:h^1\longrightarrow H^1_0(0,\pi),\qquad
 \mathcal U^{-1}:H^1_0(0,\pi)\longrightarrow h^1.
\]
The one-dimensional Sobolev embedding and product rule give
\[
 \norm{fg}_{H^1}\le C\norm f_{H^1}\norm g_{H^1}.
\]
Thus multiplication is a bounded trilinear map $(H^1_0)^3\to H^1_0$,
since the product also vanishes at the endpoints.

Set $u=\mathcal Uq$ and $v=\mathcal U\bar q$, with $q,\bar q$
independent.  Differentiating the continuous quartic integral gives
\[
 (\partial_{\bar q}P_4)_n
   =\kappa\int_0^\pi u^2v\phi_n\dd x,
 \qquad
 (\partial_qP_4)_n
   =\kappa\int_0^\pi uv^2\phi_n\dd x.
\]
These formulas first hold for variations in individual coordinates and
then for arbitrary $h^1$ variations by continuity.  In particular,
\[
 X_{P_4}
 =\left(-i\kappa\mathcal U^{-1}(u^2v),\,
          i\kappa\mathcal U^{-1}(uv^2)\right).
\]
Both components lie in $h^1$, and the multiplication estimate yields
\[
 \norm{X_{P_4}}_{h^1\times h^1}
 \le C_\omega\bigl(\norm u_{H^1}^2\norm v_{H^1}
                  +\norm u_{H^1}\norm v_{H^1}^2\bigr)
 \le C_\omega(\norm q_1+\norm{\bar q}_1)^3.
\]
The displayed formulas are compositions of bounded complex linear and
trilinear maps.  They define continuous homogeneous cubic polynomials,
hence holomorphic maps in the independent coordinates.  Restricting to
$\bar q=\overline q$ proves real analyticity on the original phase space.
\end{proof}

\subsection{Divisors with at most two normal indices}
Fix a finite nonempty tangential set $J$ and put
\[
 \mathcal N=\N\setminus J.
\]
For an ordered quartic index $(i,j,k,l)$ define
\[
 d_{\mathcal N}(i,j,k,l)
 =\mathbf1_{i\in\mathcal N}+\mathbf1_{j\in\mathcal N}
 +\mathbf1_{k\in\mathcal N}+\mathbf1_{l\in\mathcal N}
\]
and
\[
 \Delta_{ijkl}=\lambda_i+\lambda_j-\lambda_k-\lambda_l.
\]
A quadruple is trivial if the multisets $\{i,j\}$ and $\{k,l\}$ coincide.

The next lemma provides the divisor estimates required for normalization up to normal degree two.  Its constants are allowed to depend on the fixed path
and on $J$.  This dependence will only affect the final smallness threshold.

\begin{lemma}[Pathwise divisor structure]\label{lem:divisors}
Let
\[
 \mathcal C_J=
 \left\{\sum_{a\in J}d_a\lambda_a:
 d_a\in\Z,\ \sum_{a\in J}\abs{d_a}\le3\right\}.
\]
There is $c_{\omega,J}>0$ such that every nonzero divisor of the following
forms satisfies
\begin{align}
 \abs{\pm\lambda_n+C}
 &\ge c_{\omega,J}\langle n\rangle^2,
 &&C\in\mathcal C_J,
 \quad n\in\mathcal N,
 \label{eq:div1}\\
 \abs{\pm(\lambda_n+\lambda_m)+C}
 &\ge c_{\omega,J}(1+n^2+m^2),
 &&C\in\mathcal C_J,
 \quad n,m\in\mathcal N,
 \label{eq:divplus}\\
 \abs{\lambda_n-\lambda_m+C}
 &\ge c_{\omega,J}(1+\abs{n^2-m^2}),
 &&C\in\mathcal C_J,
 \quad n,m\in\mathcal N.
 \label{eq:divminus}
\end{align}
Every nontrivial quartic divisor of positive normal degree at most two is of one of these three types.  The normal-degree-zero case is finite.
\end{lemma}

\begin{proof}
Fix $0<\eps<1$.  Increasing the constant to include the finitely many
low modes, write
\[
 \lambda_n=n^2+r_n,\qquad |r_n|\le K n^\eps\quad(n\ge1).
\]
The set $\mathcal C_J$ is finite because $J$ is finite and its integer
coefficients have bounded sum of absolute values.  Put
$D_J=\max_{C\in\mathcal C_J}|C|$.

\medskip
\noindent\textit{One normal index.}
For either sign and every $C\in\mathcal C_J$,
\[
 |\pm\lambda_n+C|\ge n^2-Kn^\eps-D_J.
\]
For sufficiently large $n$, this is at least
$n^2/2\ge\langle n\rangle^2/4$.

\medskip
\noindent\textit{Two normal indices with the same sign.}
Let $N=\max(n,m)$.  The triangle inequality gives
\[
 |\pm(\lambda_n+\lambda_m)+C|
 \ge n^2+m^2-K(n^\eps+m^\eps)-D_J.
\]
The error divided by $n^2+m^2$ is at most
$2KN^{\eps-2}+D_JN^{-2}$, which tends to zero uniformly as $N\to\infty$.
Thus, for large $N$, the lower bound is at least
$(1+n^2+m^2)/4$, including the case $n=m$.

\medskip
\noindent\textit{Two normal indices with opposite signs.}
Suppose first that $n\ne m$.  Since these are positive integers,
\[
 |n^2-m^2|=|n-m|(n+m)\ge N.
\]
It follows that
\[
 \frac{|r_n-r_m+C|}{|n^2-m^2|}
 \le 2KN^{\eps-1}+D_JN^{-1}\longrightarrow0.
\]
The convergence is uniform in both indices, including neighboring
indices.  For sufficiently large $N$,
\[
 |\lambda_n-\lambda_m+C|
 \ge\frac12|n^2-m^2|
 \ge\frac14(1+|n^2-m^2|).
\]
If $n=m$, the divisor is the fixed constant $C$ and the weight in
\eqref{eq:divminus} equals $1$.

In all three cases, the remaining nonzero divisor-to-weight ratios
range over a finite set and have a positive minimum.  Taking the
minimum of these ratios and $1/4$ proves the three estimates with a
common constant $c_{\omega,J}>0$.

Finally, a quartic divisor of normal degree one contains one normal
eigenvalue and three tangential eigenvalues.  At normal degree two, its
two normal eigenvalues occur either with the same sign or with opposite
signs.  The remaining tangential combination belongs to $\mathcal C_J$.
Every nontrivial quadruple has a nonzero divisor by
\eqref{eq:all-integer-nr}.  At normal degree zero all four indices lie
in the finite set $J$, so only finitely many divisors occur.
\end{proof}

\subsection{Weighted divided-matrix estimates}

For terms containing two normal variables, division by a small divisor
produces an infinite matrix.  The entrywise estimates must imply boundedness of both the matrix and its adjoint on the energy space. The following two summability estimates establish these operator bounds.

\begin{lemma}[Weighted kernel sums]\label{lem:kernel-sums}
One has
\begin{align}
 S_+&:=\sum_{n,m\ge1}
 \frac{\langle n\rangle^2}
 {\langle m\rangle^2(1+n^2+m^2)^2}<\infty,
 \label{eq:Splus}\\
 S_-&:=\sum_{\substack{n,m\ge1\\n\ne m}}
 \frac{\langle n\rangle^2}
 {\langle m\rangle^2(1+\abs{n^2-m^2})^2}<\infty.
 \label{eq:Sminus}
\end{align}
The same statements hold after deleting finitely many indices.
\end{lemma}

\begin{proof}
We use $n^2\le\langle n\rangle^2\le2n^2$ for $n\ge1$.

For $S_+$, first consider $n\le m$.  Since $1+n^2+m^2\ge m^2$,
\[
 \frac{\langle n\rangle^2}
      {\langle m\rangle^2(1+n^2+m^2)^2}
 \le \frac{2n^2}{m^6}\le\frac2{m^4}.
\]
There are $m$ possible values of $n$ in this region, giving the bound
$2\sum_{m\ge1}m^{-3}<\infty$.  If $n>m$, the same summand is at most
$2/(m^2n^2)$, and
\[
 \sum_{m\ge1}\frac1{m^2}\sum_{n>m}\frac1{n^2}
 \le C\sum_{m\ge1}\frac1{m^3}<\infty.
\]

For $S_-$, divide the off-diagonal pairs into three regions.  If
$m/2\le n\le2m$ and $n\ne m$, the ratio
$\langle n\rangle^2/\langle m\rangle^2$ is bounded, while
$|n^2-m^2|\ge m|n-m|$.  The contribution of this region is at most
\[
 C\sum_{m\ge1}\frac1{m^2}
       \sum_{\substack{n\ge1\\ n\ne m}}\frac1{|n-m|^2}
 \le 2C\left(\sum_{m\ge1}m^{-2}\right)
         \left(\sum_{j\ge1}j^{-2}\right)<\infty.
\]
If $n>2m$, then $n^2-m^2\ge3n^2/4$, so the summand is bounded by
$C/(m^2n^2)$.  Its sum is finite by the preceding tail estimate.
If $m>2n$, then $m^2-n^2\ge3m^2/4$, and the summand is at most
$Cn^2/m^6$.  Summing first in $n$ gives
\[
 C\sum_{m\ge1}m^{-6}\sum_{1\le n<m/2}n^2
 \le C\sum_{m\ge1}m^{-3}<\infty.
\]
Deleting indices can only decrease these nonnegative sums.
\end{proof}

The isometry $(\mathcal Dz)_n=\langle n\rangle z_n$ from $h^1$ to
$\ell^2$ conjugates a matrix $K$ to
\[
 \widetilde K_{nm}=\frac{\langle n\rangle}{\langle m\rangle}K_{nm}.
\]
Consequently, if $\widetilde K$ is Hilbert--Schmidt, then
\begin{equation}\label{eq:HS-to-h1}
 \norm K_{\mathcal L(h^1,h^1)}
 =\norm{\widetilde K}_{\mathcal L(\ell^2,\ell^2)}
 \le\norm{\widetilde K}_{\HS},
\end{equation}
with the operators initially defined on finitely supported sequences
and extended by density.

\begin{lemma}[Divided multiplication matrices]\label{lem:divided-matrices}
Fix $C\in\mathcal C_J$ and $a,b\in J$, put
$g=\phi_a\phi_b\in H^1\subset L^\infty$, and define
\[
 M_{nm}=\int_0^\pi g(x)\phi_n(x)\phi_m(x)\dd x,
 \qquad n,m\in\mathcal N.
\]
Then:
\begin{enumerate}[label=\textup{(\roman*)}]
\item If
\[
 K^+_{nm}=\frac{M_{nm}}{\lambda_n+\lambda_m+C}
\]
and all denominators are nonzero, then $K^+$ and its $\ell^2$ adjoint are bounded on $h^1$.
\item If, for $n\ne m$,
\[
 K^-_{nm}=\frac{M_{nm}}{\lambda_n-\lambda_m+C},
\]
then the off-diagonal matrix $K^-_{\mathrm{off}}$ and its adjoint are bounded on $h^1$.
\item If the diagonal denominator is a fixed nonzero constant, the diagonal part is a bounded multiplier on $h^1$.
\end{enumerate}
\end{lemma}

\begin{proof}
The eigenfunctions are real and $L^2$-normalized, so
\[
 |M_{nm}|\le\norm g_\infty\int_0^\pi|\phi_n\phi_m|\dd x
 \le\norm g_\infty\norm{\phi_n}_{L^2}\norm{\phi_m}_{L^2}
 =\norm g_\infty.
\]
Also $M_{nm}=M_{mn}$ and these entries are real.

In part~(i), \cref{lem:divisors} gives
\[
 |K^+_{nm}|\le
 \frac{C_{\omega,J}\norm g_\infty}{1+n^2+m^2}.
\]
Multiplying by $\langle n\rangle/\langle m\rangle$, squaring, and
summing bounds $\norm{\widetilde K^+}_{\HS}^2$ by
$C_{\omega,J}\norm g_\infty^2 S_+$.  This is finite.
The same entry bound holds for the transpose, because the denominator
is symmetric in $n,m$; thus the transpose is also bounded on $h^1$.

For part~(ii), the denominators with $n\ne m$ are nonzero by exact
nonresonance, since $C$ contains only tangential eigenvalues.  We obtain
\[
 |K^-_{nm}|\le
 \frac{C_{\omega,J}\norm g_\infty}{1+|n^2-m^2|}\quad(n\ne m).
\]
The same calculation with $S_-$ proves the $h^1$ bound.
For the transposed entry the denominator is
$\lambda_m-\lambda_n+C=-(\lambda_n-\lambda_m-C)$.
Since $-C\in\mathcal C_J$, it satisfies the same estimate.  Thus the
transpose also acts boundedly on $h^1$.

These transposes coincide with the $\ell^2$ adjoints. To verify that the matrices also define bounded $\ell^2$ operators, put
$a_{nm}=(1+n^2+m^2)^{-2}$, or
$a_{nm}=(1+|n^2-m^2|)^{-2}$ off the diagonal.  In both cases
$a_{nm}=a_{mn}$ and
\[
 \sum_{n,m}a_{nm}
 \le\frac12\sum_{n,m}
   \left(\frac{\langle n\rangle^2}{\langle m\rangle^2}
        +\frac{\langle m\rangle^2}{\langle n\rangle^2}\right)a_{nm}
 <\infty.
\]
Thus each matrix is Hilbert--Schmidt on $\ell^2$, and its adjoint has
the transposed entries, which are real.  

For part~(iii), $|K_{nn}|\le\norm g_\infty/|C|$ gives the multiplier
bound on $h^1$, also for the adjoint.
\end{proof}

\subsection{Construction of the normal form}
We remove the nonresonant quartic monomials of normal degree at most two.
After the action--angle substitution, this normal degree is unchanged,
so the unnormalized quartic terms of normal degree three and four have
zero Taylor jet of weighted degree at most two at $(y,z,\bar z)=0$.
The construction therefore requires no inversion of divisors with three
or four unrestricted normal indices.

Let $\Pi_J$ and $\Pi_{\mathcal N}$ be the coordinate projections in the eigenbasis.  They are bounded on $h^1$.  Write
\[
 u=u_T+u_N,
 \qquad v=v_T+v_N,
\]
where $u_T=\Pi_Ju$, $u_N=\Pi_{\mathcal N}u$, and similarly for $v$.
We define the normal-degree components directly in physical space.  This definition does not require prior convergence of the four-index coefficient expansion in the vector-field norm:
\begin{equation}\label{eq:P4-r}
 P_4^{(r)}
 =\frac\kappa2
 \sum_{\substack{0\le\alpha,\beta\le2\\\alpha+\beta=r}}
 \binom2\alpha\binom2\beta
 \int_0^\pi
 u_T^{2-\alpha}u_N^\alpha
 v_T^{2-\beta}v_N^\beta\dd x,
 \qquad 0\le r\le4.
\end{equation}
Each $P_4^{(r)}$ is a continuous quartic Hamiltonian with analytic cubic vector field, and
\[
 P_4=\sum_{r=0}^4P_4^{(r)}.
\]
Set
\[
 P_4^{\le2}=P_4^{(0)}+P_4^{(1)}+P_4^{(2)},
 \qquad
 P_4^{\ge3}=P_4^{(3)}+P_4^{(4)}.
\]

On finitely supported sequences,
\begin{equation}\label{eq:P-coeff}
 P_4=\sum_{i,j,k,l\ge1}P_{ijkl}q_iq_j\bar q_k\bar q_l,
 \qquad
 P_{ijkl}=\frac\kappa2\int_0^\pi\phi_i\phi_j\phi_k\phi_l\dd x.
\end{equation}
By \cref{thm:universal-nr}, the only resonant quartic monomials are the trivial ones.  Put $I_a=q_{j_a}\bar q_{j_a}$ and define directly
\begin{align}
 Z_{TT}
 &=\frac\kappa2\sum_{a=1}^bC_{j_aj_a}I_a^2
 +2\kappa\sum_{1\le a<c\le b}C_{j_aj_c}I_aI_c,
 \label{eq:ZTT}\\
 Z_{TN}
 &=2\kappa\sum_{a=1}^b\sum_{n\in\mathcal N}
 C_{j_an}I_aq_n\bar q_n.
 \label{eq:ZTN}
\end{align}
Equivalently,
\[
 Z_{TT}=\frac12\ip{A_JI}{I}_{\R^b}
\]
and the normal-frequency coefficient vectors are
\begin{equation}\label{eq:Bn}
 B_n=(2\kappa C_{n,j_1},\dots,2\kappa C_{n,j_b}),
 \qquad n\in\mathcal N.
\end{equation}
For fixed $J$,
\begin{equation}\label{eq:Bn-bound}
 \sup_{n\in\mathcal N}\abs{B_n}
 \le C_{\omega,J},
\end{equation}
because $C_{n,j_a}\le\norm{\phi_{j_a}}_{L^\infty}^2$.

The uniform bound \eqref{eq:Bn-bound} ensures that $Z_{TN}$ has an
analytic cubic vector field.  Set
\[
 Z_4^J=Z_{TT}+Z_{TN},
 \qquad
 Q_4^J=P_4^{\le2}-Z_4^J.
\]
Define initially on finitely supported sequences
\begin{equation}\label{eq:F4}
 F_4^J
 =-\sum_{\substack{d_{\mathcal N}(i,j,k,l)\le2\\
 \{i,j\}\ne\{k,l\}}}
 \frac{P_{ijkl}}{i\Delta_{ijkl}}
 q_iq_j\bar q_k\bar q_l.
\end{equation}

\begin{proposition}[Analyticity of the normal-form generator]\label{prop:F-analytic}
The expression \eqref{eq:F4} extends uniquely to a real analytic quartic Hamiltonian near the origin of $h^1\times h^1$.  Its vector field is cubic analytic and
\begin{equation}\label{eq:F-cubic}
 \norm{X_{F_4^J}(q,\bar q)}_{h^1\times h^1}
 \le C_{\omega,J}(\norm q_1+\norm{\bar q}_1)^3.
\end{equation}
Moreover, in the sense of continuous quartic Hamiltonians on $h^1\times h^1$,
\begin{equation}\label{eq:homological}
 \{H_2,F_4^J\}+Q_4^J=0.
\end{equation}
\end{proposition}

\begin{proof}
Write $w=(q_{j_a})_{a=1}^b$ and $z=(q_n)_{n\in\mathcal N}$, and
use independent barred variables.  Set
$r=\norm q_1+\norm{\bar q}_1$.  Since $J$ is finite, tangential
Euclidean norms are bounded by $C_Jr$, and all normal $h^1$ norms are
bounded by $r$.  We separate the expression according to its normal
degree.  There are only finitely many choices of tangential indices,
signs, and positions of the normal variables.

\medskip
\noindent\textit{Step 1: terms of normal degree zero or one.}
At normal degree zero, the sum is finite and every denominator is
nonzero by \eqref{eq:all-integer-nr}.  It defines a quartic polynomial
with cubic vector field bounded by $C_{\omega,J}r^3$.

For normal degree one, fix the three tangential indices and the
position of the normal variable.  Up to fixed scalar factors, its
coefficient sequence has the form
\[
 f_n=\frac{g_n}{\pm\lambda_n+C},\qquad
 g_n=\int_0^\pi g(x)\phi_n(x)\dd x,\qquad
 g=\phi_a\phi_b\phi_c,\quad a,b,c\in J.
\]
Here $g\in H^1_0\subset L^2$.  The one-normal divisor bound and Bessel's
inequality give
\[
 \norm f_1^2
 =\sum_{n\in\mathcal N}\langle n\rangle^2
             \frac{|g_n|^2}{|\pm\lambda_n+C|^2}
 \le C_{\omega,J}\sum_{n\in\mathcal N}
                    \frac{|g_n|^2}{\langle n\rangle^2}
 \le C_{\omega,J}\norm g_{L^2}^2.
\]
The complex linear functional $\ell_f(z)=\sum_nf_nz_n$ is continuous,
since $|\ell_f(z)|\le\norm f_{\ell^2}\norm z_{\ell^2}$.
Its coordinate gradient is the sequence $f$, which belongs to $h^1$.
Multiplication by the three tangential factors therefore yields a
quartic Hamiltonian with vector field bounded by $C_{\omega,J}r^3$.

\medskip
\noindent\textit{Step 2: terms of normal degree two.}
Fix the two tangential indices and the locations of the two normal
variables.  The numerator is a multiplication-matrix entry
$\int\phi_a\phi_b\phi_n\phi_m$.  If both normal variables are unbarred
or both are barred, their eigenvalues have the same sign in the divisor,
so part~(i) of \cref{lem:divided-matrices} applies, after changing an
overall sign if necessary.  For one barred and one unbarred normal
variable, part~(ii) applies when $n\ne m$.  When $n=m$, the normal eigenvalues
cancel: a zero tangential divisor gives a trivial monomial, already
removed from the sum, and every other diagonal term has a fixed nonzero
tangential denominator and is covered by part~(iii).

For a matrix $K$ of one of these types, put
\[
 \mathcal B_K(z,\zeta)=\sum_{n,m\in\mathcal N}K_{nm}z_m\zeta_n
                     =\sum_{n\in\mathcal N}(Kz)_n\zeta_n.
\]
This is a complex bilinear pairing, with no complex conjugation; the
variables in the complexification must remain independent.  Since
$K:h^1\to h^1$ is bounded,
\[
 |\mathcal B_K(z,\zeta)|
 \le\norm{Kz}_{\ell^2}\norm\zeta_{\ell^2}
 \le C\norm z_1\norm\zeta_1.
\]
The gradients with respect to $\zeta$ and $z$ are $Kz$ and
$K^{\mathsf T}\zeta$, respectively.  Both lie in $h^1$ by the divided
matrix estimates.  The same conclusion holds if $K$ has a fixed complex
scalar factor.  If the two variables coincide, differentiation gives
$(K+K^{\mathsf T})z$, which is also bounded in $h^1$.

Multiplying this quadratic form by its two tangential factors gives a
quartic Hamiltonian.  A normal derivative is bounded by
$C(|w|+|\bar w|)^2(\norm z_1+\norm{\bar z}_1)$, and a tangential derivative by
$C(|w|+|\bar w|)(\norm z_1+\norm{\bar z}_1)^2$.
Thus every component is bounded by
$C_{\omega,J}r^3$.  The bilinear forms are the limits of their
finite-coordinate truncations: apply boundedness of $K$ to the
convergent truncations of $z$ and $\zeta$ in $h^1$.

\medskip
\noindent\textit{Step 3: extension, analyticity, and reality.}
Summing the finitely many classes of terms from Steps~1 and~2 defines a continuous
homogeneous quartic polynomial on $h^1\times h^1$.  Its coordinate
gradients are continuous homogeneous cubic polynomials taking values
in $h^1$.  In particular the Hamiltonian and its vector field are
holomorphic on the complexification, and the preceding estimates give
\eqref{eq:F-cubic}.  This extension agrees with \eqref{eq:F4} on
finitely supported sequences and is unique by density.

To check reality, exchange $(i,j)$ with $(k,l)$.
The real coefficient $P_{ijkl}$ is unchanged, while
$\Delta_{klij}=-\Delta_{ijkl}$.  Consequently
\[
 \overline{\left(-\frac{P_{ijkl}}{i\Delta_{ijkl}}\right)}
 =-\frac{P_{klij}}{i\Delta_{klij}}.
\]
On $\bar q=\overline q$, the corresponding monomials are conjugate.
The truncated sums are therefore real, and so is their continuous
extension.  The vector field preserves the relation $\bar q=\overline q$.

\medskip
\noindent\textit{Step 4: the homological equation.}
For $M_{ijkl}=q_iq_j\bar q_k\bar q_l$, use
$(H_2)_{q_n}=\lambda_n\bar q_n$ and
$(H_2)_{\bar q_n}=\lambda_nq_n$ in \eqref{eq:Poisson}.  Counting each
occurrence of an index, including repeated indices, gives
\begin{equation}\label{eq:H2-monomial}
 \{H_2,M_{ijkl}\}=i(\lambda_i+\lambda_j-\lambda_k-\lambda_l)M_{ijkl}
                  =i\Delta_{ijkl}M_{ijkl}.
\end{equation}
Thus the coefficient in \eqref{eq:F4} produces $-P_{ijkl}M_{ijkl}$.
The sum of these nontrivial monomials is precisely $Q_4^J$, so
\eqref{eq:homological} holds on finitely supported sequences.

It remains to extend this identity to the energy space.  Since
$|\lambda_n|\le C_\omega\langle n\rangle^2$, Cauchy--Schwarz yields
\[
 |\dd H_2(q,\bar q)[v,\bar v]|
 \le C_\omega(\norm q_1+\norm{\bar q}_1)
                 (\norm v_1+\norm{\bar v}_1).
\]
Together with \eqref{eq:F-cubic}, this shows that $\dd H_2[X_{F_4^J}]$ is a continuous quartic functional.  Both sides of the homological equation are continuous on $h^1\times h^1$,
so density proves the identity everywhere.  In particular, the bracket has the analytic cubic vector field
$-X_{Q_4^J}$.
\end{proof}

Let $\Phi_J^t$ denote the local flow of $X_{F_4^J}$.

\begin{proposition}[Partial quartic Birkhoff normal form]\label{prop:partial-BNF}
There is a neighborhood $U_{\omega,J}$ of the origin in $h^1\times h^1$ such that $\Phi_J^t$ is defined for $\abs t\le1$, is real analytic and symplectic, and
\begin{equation}\label{eq:BNF}
 H\circ\Phi_J^1
 =H_2+Z_4^J+P_4^{\ge3}+R_6^J,
\end{equation}
where $X_{R_6^J}$ is real analytic, vanishes to order five at the origin, and
\begin{equation}\label{eq:R6-bound}
 \norm{X_{R_6^J}(q,\bar q)}_{h^1\times h^1}
 \le C_{\omega,J}(\norm q_1+\norm{\bar q}_1)^5
\end{equation}
on a smaller neighborhood.
\end{proposition}

\begin{proof}
Abbreviate $F=F_4^J$, $Q=Q_4^J$, and $\Phi^t=\Phi_J^t$.
In this proof write $w=(q,\bar q)$ and
$\norm w=\norm q_1+\norm{\bar q}_1$.

\medskip
\noindent\textit{Step 1: the flow and its differential.}
The preceding proposition gives
\[
 \norm{X_F(w)}\le C\norm w^3,\qquad
 \norm{DX_F(w)}\le C\norm w^2.
\]
The second estimate follows by differentiating the bounded cubic
polynomial.  The local existence theorem for analytic ordinary
differential equations in a Banach space applies to $X_F$.
To obtain a common time interval, choose $r_0>0$ small enough that
$8Cr_0^2<1$.  For initial data $\norm w\le r_0$, the integral equation
and a continuation argument give
\[
 \norm{\Phi^t(w)}\le2\norm w,\qquad
 \norm{\Phi^t(w)-w}\le C\norm w^3,
 \qquad |t|\le1,
\]
after increasing $C$ if necessary.
Analytic dependence on initial data gives a real analytic flow, with a
holomorphic extension to a smaller complex ball.  The variational equation
\[
 \frac{\dd}{\dd t}D\Phi^t(w)
 =DX_F(\Phi^t(w))D\Phi^t(w),\qquad D\Phi^0(w)=\Id,
\]
and Gronwall's inequality give
\[
 \norm{D\Phi^t(w)-\Id}\le C\norm w^2.
\]
On a smaller ball these differentials are invertible with uniformly
bounded inverses.  The flow identity also provides the inverse map
$\Phi^{-t}$ on the image.  Reality of $F$ implies that the flow preserves $\bar q=\overline q$.

\medskip
\noindent\textit{Step 2: preservation of the weak symplectic form.}
For fixed tangent vectors $U,V$, put
$U_t=D\Phi^t(w)U$ and $V_t=D\Phi^t(w)V$.
Differentiating $\Om_{\mathrm{symp}}(X_F,\cdot)=\dd F$ in a constant
direction gives
$\Om_{\mathrm{symp}}(DX_F U,V)=D^2F[U,V]$.
The variational equation and skew symmetry of $\Om_{\mathrm{symp}}$
therefore yield
\[
 \begin{aligned}
 \frac{\dd}{\dd t}\Om_{\mathrm{symp}}(U_t,V_t)
 &=\Om_{\mathrm{symp}}(DX_F U_t,V_t)
   +\Om_{\mathrm{symp}}(U_t,DX_F V_t)\\
 &=D^2F[U_t,V_t]-D^2F[V_t,U_t]=0.
 \end{aligned}
\]
Thus $(\Phi^t)^*\Om_{\mathrm{symp}}=\Om_{\mathrm{symp}}$.

\medskip
\noindent\textit{Step 3: an exact integral formula for the remainder.}
Using the scalar interpretation of the homological equation, the
chain rule along $\Phi^t$ gives
\[
 \frac{\dd}{\dd t}(H_2\circ\Phi^t)=-Q\circ\Phi^t,
 \qquad
 \frac{\dd}{\dd t}(P_4\circ\Phi^t)=\{P_4,F\}\circ\Phi^t.
\]
Applying the chain rule to $Q$ and integrating once more gives
\begin{equation}\label{eq:Lie-Taylor}
 H_2\circ\Phi^1
 =H_2-Q-\int_0^1(1-t)\{Q,F\}\circ\Phi^t\dd t.
\end{equation}
Combining this with the first-order integral formula for $P_4$ gives
\[
 H\circ\Phi^1=H_2+P_4-Q+R_6^J,
\]
where the exact remainder is
\[
 R_6^J
 =\int_0^1\{P_4,F\}\circ\Phi^t\dd t
   -\int_0^1(1-t)\{Q,F\}\circ\Phi^t\dd t.
\]
Since $P_4-Q=Z_4^J+P_4^{\ge3}$, this is \eqref{eq:BNF}.

\medskip
\noindent\textit{Step 4: the vector field of the remainder.}
If $G$ and $F$ have bounded Hamiltonian vector fields, differentiation
of $\{G,F\}=\dd G[X_F]$, using symmetry of the Hessians, gives
\[
 X_{\{G,F\}}=DX_G X_F-DX_F X_G.
\]
More explicitly, pairing the right-hand side with a vector $V$ by
$\Om_{\mathrm{symp}}$ gives
$D^2G[X_F,V]+\dd G[DX_F V]=\dd(\dd G[X_F])[V]$.
For $G=P_4$ and $G=Q$,
both vector fields are cubic and their differentials are quadratic.
Thus the two brackets are sextic Hamiltonians with analytic quintic
vector fields satisfying
\[
 \norm{X_{\{G,F\}}(w)}\le C\norm w^5.
\]

We next estimate the vector fields after composition with the flow.  Symplecticity and the chain rule show that for either bracket
$S=\{G,F\}$,
\[
 X_{S\circ\Phi^t}(w)
   =(D\Phi^t(w))^{-1}X_S(\Phi^t(w)).
\]
Indeed, pairing this vector with $V$ equals
$\Om_{\mathrm{symp}}(X_S(\Phi^t(w)),D\Phi^t(w)V)
=\dd(S\circ\Phi^t)(w)[V]$.
Step~1 bounds the inverse differential and the size of $\Phi^t(w)$,
so $\norm{X_{S\circ\Phi^t}(w)}\le C\norm w^5$, uniformly for
$0\le t\le1$.  These maps are analytic on a common smaller complex
ball, with uniform bounds that justify integration and differentiation
under the integral sign.
The displayed formula for $R_6^J$ consequently proves
\eqref{eq:R6-bound}.  The same bound implies that its Taylor terms of degree less than five
vanish.
\end{proof}

\section{KAM construction}\label{sec:KAM}

\subsection{An abstract KAM theorem}\label{sec:Poschel}
We record the consequence of P\"oschel's Theorems A and B and Corollary C
that will be used below~\cite{Poschel1996}.  We take the exponential
sequence weight in that paper to be zero.  The complex $(z,\bar z)$
notation is the standard complexification of P\"oschel's real normal
coordinates; the corresponding weighted norms are equivalent.  For $d>1$, P\"oschel's regularity assumption allows $\bar p=p$; thus the
bounded case $p=\bar p=1$ used here is included.  Let
$\Pi\subset\R^b$ be compact with positive Lebesgue measure and consider
\[
 N=\ip{\omega(\xi)}y+\sum_{r\ge1}\Omega_r(\xi)z_r\bar z_r
\]
on $\T^b\times\R^b\times h^p\times h^p$, where
$h^p=\{z:\norm z_p^2=\sum_{r\ge1}\langle r\rangle^{2p}|z_r|^2<\infty\}$
for $p\ge0$.  For integer multi-indices, $|k|=\sum_a|k_a|$ and
$|l|=\sum_r|l_r|$.  A lipeomorphism is a bijection onto its image
for which both the map and its inverse are Lipschitz.
For a Lipschitz map $f$ on $\Pi$, write $\norm f_\Pi^{\mathcal L}$ for
its supremum norm plus Lipschitz seminorm.  Put
\[
 \norm w_{-\delta}=\sup_{r\ge1}r^{-\delta}\abs{w_r}.
\]

\begin{theorem}[P\"oschel, specialized bounded-perturbation form]\label{thm:Poschel}
Assume the following.
\begin{enumerate}[label=\textup{(P\arabic*)}]
\item The map $\xi\mapsto\omega(\xi)$ is a lipeomorphism onto its image.  For every $(k,l)\ne0$ with $k\in\Z^b$, $l\in\Z^{(\N)}$, and $\abs l\le2$, the zero set of
\[
 \xi\longmapsto k\cdot\omega(\xi)+l\cdot\Omega(\xi)
\]
has Lebesgue measure zero; moreover $l\cdot\Omega(\xi)\ne0$ on $\Pi$ whenever $1\le\abs l\le2$.
\item There are $d>1$, $0\le\delta<d-1$, and a fixed parameter-independent sequence
\[
 \bar\Omega_r=r^d+\sum_{j=1}^{s_0}a_jr^{d_j},
 \qquad d>d_1>\cdots>d_{s_0},
\]
where $s_0\ge0$, the coefficients and exponents are fixed real numbers,
and finitely many entries may be assigned arbitrary fixed values, such that $\widetilde\Omega=\Omega-\bar\Omega$ is Lipschitz from $\Pi$ to $\ell^\infty_{-\delta}$.
\item The perturbation $P$ is real analytic in the phase variables, Lipschitz in $\xi$, and $X_P$ maps the phase space with index $p$ analytically into one with index $\bar p\ge p$.
\end{enumerate}
Assume bounds
\begin{equation}\label{eq:ML-bounds}
 \norm\omega_\Pi^{\mathcal L}
 +\norm{\widetilde\Omega}_{-\delta,\Pi}^{\mathcal L}
 \le M,
 \qquad
 \norm{\omega^{-1}}_{\omega(\Pi)}^{\mathcal L}\le L.
\end{equation}
Fix a strip width $s>0$ and choose
$\tau\ge b+1+2/(d-1)$, as in equation~(22) of
\cite{Poschel1996}, so that the geometric resonance sums converge.  On
\[
 D(s,r)=\{\abs{\Im x}<s,\ \abs y<r^2,\ \norm z_p+\norm{\bar z}_p<r\}
\]
use
\begin{equation}\label{eq:scaled-norm-general}
 \norm X_r
 =\sup_{D(s,r)}
 \left(\abs{X^x}+r^{-2}\abs{X^y}
 +r^{-1}(\norm{X^z}_{\bar p}+\norm{X^{\bar z}}_{\bar p})\right).
\end{equation}
There is a constant $\gamma>0$ such that, if $0<\alpha\le1$ and
\begin{equation}\label{eq:Poschel-smallness}
 \eps_P:=\norm{X_P}_r
 +\frac\alpha M\norm{X_P}^{\Lip}_r
 \le\gamma\alpha,
\end{equation}
then $N+P$ has a Cantor family of real-analytic, linearly stable invariant $b$-tori over a set $\Pi_\alpha\subset\Pi$.

If the unperturbed frequency maps are affine and $d>1$, then, after
replacing $\gamma$ by the smaller constant from Corollary C and taking
$\alpha>0$ sufficiently small,
\begin{equation}\label{eq:Poschel-measure}
 \abs{\Pi\setminus\Pi_\alpha}
 \le C(\operatorname{diam}\Pi)^{b-1}\alpha.
\end{equation}
The constants depend on $b,p,\bar p,d,\delta,\tau,s,M,L$, the fixed
asymptotic data, and the finitely many affine nondegeneracy constants that
occur in Theorem~B.
\end{theorem}

The estimate in \eqref{eq:Poschel-measure} is the $d>1$ affine-frequency
case discussed immediately after Corollary~C in~\cite{Poschel1996}.  Since
our parameter domains shrink to the origin, the uniformity of its constant
is established in \cref{lem:uniform-shrinking-boxes}, after verification of the frequency hypotheses.

\subsection{Verification of the hypotheses and proof}

\paragraph{Action--angle variables and frequencies.}
Fix $J=\{j_1,\dots,j_b\}$ and introduce, in the complexification,
\begin{equation}\label{eq:action-angle}
 q_{j_a}=\sqrt{\xi_a+y_a}\,e^{-ix_a},
 \qquad
 \bar q_{j_a}=\sqrt{\xi_a+y_a}\,e^{ix_a},
 \qquad a=1,\dots,b,
\end{equation}
while $q_n=z_n$, $\bar q_n=\bar z_n$ for $n\in\mathcal N$.  With the signs chosen in \eqref{eq:action-angle}, direct calculation gives
\begin{equation}\label{eq:AA-symplectic}
 i\,\dd q_{j_a}\wedge\dd\bar q_{j_a}
 =\dd x_a\wedge\dd y_a,
\end{equation}
so $\dot x=\partial_yH$ and $\dot y=-\partial_xH$.

Let
\[
 \widetilde H=H\circ\Phi_J^1.
\]
Discarding constants depending only on $\xi$, \eqref{eq:BNF}, \eqref{eq:ZTT}, and \eqref{eq:ZTN} give
\begin{equation}\label{eq:Nxi}
 N_\xi=\ip{\omega(\xi)}y
 +\sum_{n\in\mathcal N}\Omega_n(\xi)z_n\bar z_n,
\end{equation}
where
\begin{equation}\label{eq:frequencies}
 \omega(\xi)=\lambda_J+A_J\xi,
 \qquad
 \Omega_n(\xi)=\lambda_n+B_n\cdot\xi.
\end{equation}
The remaining perturbation is
\begin{equation}\label{eq:Pnu}
 P_\nu=P_{yy}+P_{yz^2}+P_4^{\ge3}+R_6^J,
\end{equation}
with
\begin{equation}\label{eq:Pyy-Pyz}
 P_{yy}=\frac12\ip{A_Jy}{y},
 \qquad
 P_{yz^2}=\sum_{n\in\mathcal N}(B_n\cdot y)z_n\bar z_n.
\end{equation}

\paragraph{Exact resonance nonidentities and a uniform second-Melnikov gap.}

\begin{lemma}\label{lem:global-two-mode-gap}
There is $d_\omega>0$ such that
\begin{equation}\label{eq:domega}
 \inf_{n\in\mathcal N}\abs{\lambda_n}
 \wedge
 \inf_{n,m\in\mathcal N}\abs{\lambda_n+\lambda_m}
 \wedge
 \inf_{\substack{n,m\in\mathcal N\\n\ne m}}\abs{\lambda_n-\lambda_m}
 \ge d_\omega.
\end{equation}
\end{lemma}

\begin{proof}
The three families have only finitely many elements of absolute value
at most $1$.  Indeed, $\lambda_n\to+\infty$; for sums, with
$N=\max(n,m)$,
\[
 \lambda_n+\lambda_m\ge\lambda_N+\lambda_1\longrightarrow+\infty;
\]
and for differences with $n>m$,
\[
 \lambda_n-\lambda_m\ge\lambda_n-\lambda_{n-1}\longrightarrow+\infty
\]
by \eqref{eq:gap-asymptotic}.  All the remaining finitely many values
are nonzero by \eqref{eq:all-integer-nr}.  Their absolute values,
together with $1$, therefore have a positive minimum $d_\omega$.
This argument applies to the full spectrum, so $d_\omega$ is independent
of $J$.
\end{proof}

\begin{proposition}[Verification of (P1)]\label{prop:P1}
There is $\nu_*=\nu_*(\omega,J)>0$ such that the frequencies
\eqref{eq:frequencies} satisfy (P1) on the reference box
$[0,2\nu_*]^b$, and hence on every $\Pi_\nu$ with
$0<\nu<\nu_*$.
\end{proposition}

\begin{proof}
Since $A_J$ is invertible on the fixed path,
\[
 \omega^{-1}(\eta)=A_J^{-1}(\eta-\lambda_J)
\]
on the image of any parameter box.  The frequency map and its inverse
have Lipschitz constants $\norm{A_J}$ and $\norm{A_J^{-1}}$,
independently of the size of that box.  This proves the lipeomorphism
requirement.

Fix $(k,l)\ne0$ with $|l|\le2$.  Its resonance function is
\[
 \begin{aligned}
 f_{k,l}(\xi)
 &=k\cdot\omega(\xi)+\sum_{n\in\mathcal N}l_n\Omega_n(\xi)\\
 &=k\cdot\lambda_J+\sum_{n\in\mathcal N}l_n\lambda_n
   +\left(A_J^{\mathsf T}k+\sum_{n\in\mathcal N}l_nB_n\right)\cdot\xi.
 \end{aligned}
\]
Since $J$ and $\mathcal N$ are disjoint, the constant term is a
nonzero integer spectral combination by \eqref{eq:all-integer-nr}.
Thus $f_{k,l}$ is not identically zero, and its zero set is either
empty or an affine hyperplane, hence has Lebesgue measure zero.

For the remaining requirement, $1\le|l|\le2$ means that, up to an
overall sign, $l\cdot\lambda$ is one of
\[
 \lambda_n,\quad 2\lambda_n,\quad
 \lambda_n+\lambda_m,\quad
 \lambda_n-\lambda_m\ (n\ne m).
\]
Their absolute values are at least $d_\omega$ by the preceding lemma.
Define the finite constant
\[
 D_B=\sup_{n\in\mathcal N}\sum_{a=1}^b|(B_n)_a|.
\]
For $\xi\in[0,2\nu_*]^b$,
\[
 \left|\sum_nl_nB_n\cdot\xi\right|
 \le |l|D_B|\xi|_\infty\le4D_B\nu_*.
\]
Choose, for instance,
$0<\nu_*\le\min(1,d_\omega/[8(1+D_B)])$.  The triangle inequality gives
\[
 |l\cdot\Omega(\xi)|\ge d_\omega-4D_B\nu_*\ge d_\omega/2>0
\]
uniformly in all such $l$ and all $\xi$ in the reference box.
This proves (P1) on the reference box and its restrictions.
\end{proof}

\paragraph{Spectral asymptotics and uniform frequency constants.}
Enumerate the cofinite normal set increasingly,
\[
 \mathcal N=\{n_1<n_2<\cdots\}.
\]
For all sufficiently large $r$, $n_r=r+b$; in particular $n_r\asymp r$, so the reindexed normal $h^1$ norm is equivalent to the original one.  Define a fixed parameter-independent sequence $\bar\Omega$ by setting
\[
 \bar\Omega_r=(r+b)^2
\]
for all sufficiently large $r$, and assigning arbitrary fixed values to its finitely many initial entries (for instance $n_r^2$).  Then $\bar\Omega_r=r^2+2br+b^2$ for large $r$.

\begin{proposition}[Verification of (P2)]\label{prop:P2}
For every fixed $0<\delta<1$, the normal frequencies satisfy (P2) with $d=2$.
\end{proposition}

\begin{proof}
Fix $0<\eps<\delta<1$.  By \eqref{eq:eig-asymptotic} and the chosen
sequence $\bar\Omega$,
\[
 |\lambda_{n_r}-\bar\Omega_r|\le C_{\omega,J,\eps}r^\eps,
 \qquad
 C_\lambda:=\sup_{r\ge1}r^{-\delta}
                 |\lambda_{n_r}-\bar\Omega_r|<\infty,
\]
after increasing the constant to include the finitely many initial
entries.
Let $C_B=\sup_{n\in\mathcal N}|B_n|<\infty$ and
$R_*=\sup_{\xi\in[0,2\nu_*]^b}|\xi|$.
For $\widetilde\Omega_r(\xi)
=\lambda_{n_r}-\bar\Omega_r+B_{n_r}\cdot\xi$ we have
\[
 \sup_{\xi\in[0,2\nu_*]^b}\norm{\widetilde\Omega(\xi)}_{-\delta}
 \le C_\lambda+C_BR_*.
\]
For two parameters $\xi,\eta$ in this box,
\[
 \begin{aligned}
 \norm{\widetilde\Omega(\xi)-\widetilde\Omega(\eta)}_{-\delta}
 &=\sup_{r\ge1}r^{-\delta}|B_{n_r}\cdot(\xi-\eta)|\\
 &\le C_B|\xi-\eta|.
 \end{aligned}
\]
Thus the remainder is Lipschitz into $\ell^\infty_{-\delta}$ on the
fixed reference box, and therefore on all smaller action boxes.
Since $d=2$ and $0<\delta<d-1$, this proves (P2).
\end{proof}

Set
\[
 \Pi^0=[0,2\nu_*]^b.
\]
Fix constants
\begin{equation}\label{eq:correct-M-L}
 \mathfrak M
 \ge1+\norm\omega_{\Pi^0}^{\mathcal L}
 +\norm{\widetilde\Omega}_{-\delta,\Pi^0}^{\mathcal L},
 \qquad
 \mathfrak L\ge1+\norm{\omega^{-1}}_{\omega(\Pi^0)}^{\mathcal L}.
\end{equation}
These constants depend only on $(\omega,J,\delta)$ and provide the
bounds $M,L$ in \cref{thm:Poschel} uniformly on all $\Pi_\nu\subset\Pi^0$.

\begin{lemma}[Uniformity on shrinking action boxes]
\label{lem:uniform-shrinking-boxes}
Let $\alpha_\nu>0$ satisfy $\alpha_\nu\to0$ as $\nu\downarrow0$.
For all sufficiently small $\nu$, the smallness constant in P\"oschel's
Corollary~C and the constant in the measure estimate
\eqref{eq:Poschel-measure} can be chosen independently of $\nu$ for the
domains $\Pi_\nu=[\nu,2\nu]^b$.
\end{lemma}

\begin{proof}
\medskip
\noindent\textit{Step 1: choice of constants in the geometric estimate.}
Apply the construction in P\"oschel's Theorem~B on $\Pi^0$, with the
fixed bounds $\mathfrak M,\mathfrak L$ and the fixed spectral asymptotic
data.  Its conclusion separates a finite exceptional set of resonance
labels, denoted by $\mathcal X_0$, from the remaining labels.
For $d>1$, the union of resonance zones with nonexceptional labels
satisfies a bound of the form
$C_0(\operatorname{diam}\Pi)^{b-1}\alpha$.
The domain-monotonicity statement in that theorem means that the
constant $C_0$ and the exceptional set do not increase when $\Pi^0$
is restricted to a closed subset.  In particular, on $\Pi_\nu$ we may
use $C_0$ and an exceptional set
$\mathcal X_\nu\subset\mathcal X_0$.
The lower bounds for pure normal resonances needed in that construction
are also uniform under restriction, by the proof of (P1).

The finitely many exceptional labels require separate treatment, since their contribution is not included in the preceding estimate.
If $\mathcal X_0$ is empty, Steps~2 and~3 below are unnecessary.

\medskip
\noindent\textit{Step 2: lower bounds for the exceptional affine functions near zero.}
For the affine resonance functions $f_{k,l}$ from the proof of
\cref{prop:P1}, we have $f_{k,l}(0)\ne0$ for every
$(k,l)\in\mathcal X_0$.
Since $\mathcal X_0$ is finite, the constants
\[
 c_0=\frac12\min_{(k,l)\in\mathcal X_0}|f_{k,l}(0)|>0,
 \qquad
 D_0=\max_{(k,l)\in\mathcal X_0}
              \sum_{a=1}^b|\partial_{\xi_a}f_{k,l}|<\infty
\]
are well defined.  Since the functions are affine, their derivatives are independent of $\xi$.  Choose $0<\nu_1\le\nu_*$ with $2D_0\nu_1\le c_0$; if
$D_0=0$, no additional restriction is needed.  Then
\[
 |f_{k,l}(\xi)|
 \ge |f_{k,l}(0)|-D_0|\xi|_\infty
 \ge2c_0-2D_0\nu_1\ge c_0
\]
for every label in $\mathcal X_0$ and every
$\xi\in[0,2\nu_1]^b$.

\medskip
\noindent\textit{Step 3: exclusion of the exceptional zeroth-step zones.}
Choose the smaller smallness constant from P\"oschel's Corollary~C
once, using $\mathcal X_0$.  This choice ensures that the initial
Fourier cutoff $K_0$ is at least
$\max_{(k,l)\in\mathcal X_0}|k|$.  In the exclusion scheme of that
corollary, a label with $|k|\le K_0$ can occur only at step zero.
Thus the only zones to consider for labels in
$\mathcal X_\nu\subset\mathcal X_0$ use the unperturbed frequencies:
\[
 \mathcal R^0_{k,l}(\alpha_\nu)
 =\left\{\xi\in\Pi_\nu:
 |f_{k,l}(\xi)|<\alpha_\nu\frac{\langle l\rangle_d}{A_k}\right\},
\]
where, with $\mathcal N=\{n_1<n_2<\cdots\}$,
\[
 \langle l\rangle_d=\max\left(1,
               \left|\sum_{r\ge1}r^d l_{n_r}\right|\right),
 \qquad A_k=1+|k|^\tau,
 \qquad d=2.
\]
All these weights are fixed numbers for each exceptional label.  Set
\[
 C_{\mathcal X}=\max_{(k,l)\in\mathcal X_0}
                         \frac{\langle l\rangle_d}{A_k}<\infty.
\]
For sufficiently small $\nu<\nu_1$, the hypothesis
$\alpha_\nu\to0$ gives $\alpha_\nu C_{\mathcal X}<c_0$.
The lower bound from Step~2 implies that all these zones are empty.

\medskip
\noindent\textit{Step 4: uniformity of the constants.}
The union of all remaining excluded zones has measure at most
\[
 C_0(\operatorname{diam}\Pi_\nu)^{b-1}\alpha_\nu
\]
by Step~1.  The smaller smallness constant was chosen using a single
finite set and fixed frequency bounds, so it is independent of $\nu$.
All other smallness restrictions in Theorem~B can be made on the same
reference box and remain valid for sufficiently small $\alpha_\nu$.
The constants $c_0,D_0$ and $C_{\mathcal X}$ may depend on the fixed
path and $J$, but enter only the smallness threshold for $\nu$.  This is the
uniformity required in the quenched construction.
\end{proof}

\paragraph{Scaled perturbation estimates.}
We use the domain $D(s,r)$ and the scaled norm
\eqref{eq:scaled-norm-general} with $p=\bar p=1$.
Fix
\begin{equation}\label{eq:beta-sigma}
 0<\beta<\sigma<\frac13
\end{equation}
and put
\begin{equation}\label{eq:scales}
 r_\nu=\nu^{1/2+\sigma},
 \qquad
 \alpha_\nu=\nu^{1+\beta}.
\end{equation}
For all sufficiently small $\nu$, $r_\nu^2\le\nu/4$, so the square roots in \eqref{eq:action-angle} are holomorphic on the complex phase domain, uniformly for $\xi\in\Pi_\nu$.

\begin{lemma}[Normal-degree multilinear bounds]\label{lem:normal-degree-bounds}
Let $w=(q_{j_a})_{a=1}^b$, let $z=(q_n)_{n\in\mathcal N}$, and treat conjugate variables independently.  Put
\[
 a_J=\abs w+\abs{\bar w},
 \qquad
 \zeta=\norm z_1+\norm{\bar z}_1.
\]
Then
\begin{align}
 \abs{\partial_wP_4^{\ge3}}
 +\abs{\partial_{\bar w}P_4^{\ge3}}
 &\le C_{\omega,J}\zeta^3,
 \label{eq:Pge3-tang-der}\\
 \norm{\partial_zP_4^{\ge3}}_1
 +\norm{\partial_{\bar z}P_4^{\ge3}}_1
 &\le C_{\omega,J}(a_J\zeta^2+\zeta^3).
 \label{eq:Pge3-normal-der}
\end{align}
\end{lemma}

\begin{proof}
The spectral transform and its projections give
\[
 \norm{u_T}_{H^1}+\norm{v_T}_{H^1}\le C_{\omega,J}a_J,
 \qquad
 \norm{u_N}_{H^1}+\norm{v_N}_{H^1}\le C_\omega\zeta.
\]
Expanding the binomial coefficients in \eqref{eq:P4-r} at normal
degrees three and four gives the explicit formulas
\[
 \begin{aligned}
 P_4^{(3)}&=\kappa\int_0^\pi
            (u_Tu_Nv_N^2+u_N^2v_Tv_N)\dd x,\\
 P_4^{(4)}&=\frac\kappa2\int_0^\pi u_N^2v_N^2\dd x.
 \end{aligned}
\]
Tangential differentiation gives
\[
 \partial_{w_a}P_4^{\ge3}
    =\kappa\int_0^\pi\phi_{j_a}u_Nv_N^2\dd x,
 \qquad
 \partial_{\bar w_a}P_4^{\ge3}
    =\kappa\int_0^\pi\phi_{j_a}u_N^2v_N\dd x.
\]
Each integral is bounded by $C_{\omega,J}\zeta^3$.
There are only $b$ tangential components, proving
\eqref{eq:Pge3-tang-der}.

To estimate the normal gradients in $h^1$, differentiate in physical
space first.  The functions paired with variations of $u_N$ and $v_N$
are, respectively,
\[
 \begin{aligned}
 G_u&=\kappa(u_Tv_N^2+2u_Nv_Tv_N+u_Nv_N^2),\\
 G_v&=\kappa(2u_Tu_Nv_N+u_N^2v_T+u_N^2v_N).
 \end{aligned}
\]
All products belong to $H^1_0$.  By the trilinear multiplication
estimate proved in \cref{prop:P4-analytic},
\[
 \norm{G_u}_{H^1}+\norm{G_v}_{H^1}
 \le C_{\omega,J}(a_J\zeta^2+\zeta^3).
\]
Their normal spectral coefficients are exactly
$\partial_zP_4^{\ge3}$ and $\partial_{\bar z}P_4^{\ge3}$.
Applying the bounded inverse spectral transform and normal projection
gives \eqref{eq:Pge3-normal-der} in $h^1$.
\end{proof}

\begin{lemma}[Quartic perturbation estimate]\label{lem:quartic-estimate}
On $D(s,r_\nu)\times\Pi_\nu$,
\begin{align}
 \norm{X_{P_{yy}+P_{yz^2}+P_4^{\ge3}}}_{r_\nu}
 &\le C_{\omega,J}\nu^{1+\sigma},
 \label{eq:quartic-sup}\\
 \norm{X_{P_{yy}+P_{yz^2}+P_4^{\ge3}}}^{\Lip}_{r_\nu}
 &\le C_{\omega,J}\nu^\sigma.
 \label{eq:quartic-lip}
\end{align}
\end{lemma}

\begin{proof}
Write $r=r_\nu$ and take $0<\nu<1$ sufficiently small that
$r^2\le\nu/4$.  Constants below are independent of $\nu$; the strip
width $s$ is fixed.

\medskip
\noindent\textit{Step 1: the terms polynomial in $y,z,\bar z$.}
The fields of $P_{yy}+P_{yz^2}$ have components
\[
 \begin{aligned}
 X^x&=A_Jy+\sum_{n\in\mathcal N}B_nz_n\bar z_n,
 &X^y&=0,\\
 X^z_n&=-i(B_n\cdot y)z_n,
 &X^{\bar z}_n&=i(B_n\cdot y)\bar z_n.
 \end{aligned}
\]
The sum in $X^x$ converges absolutely by Cauchy--Schwarz and the uniform
bound on $B_n$.  Since $|y|<r^2$ and
$\norm z_1+\norm{\bar z}_1<r$, we obtain
\[
 |X^x|\le Cr^2,\qquad
 \norm{X^z}_1+\norm{X^{\bar z}}_1\le Cr^3.
\]
Thus their scaled norm is at most $Cr^2$.
These expressions do not depend on $\xi$, so their parameter
Lipschitz seminorm is zero.

\medskip
\noindent\textit{Step 2: the remaining quartic terms.}
For every tangential index,
\[
 \frac34\nu\le|\xi_a+y_a|\le\frac94\nu,
 \qquad |e^{\pm ix_a}|\le e^s.
\]
The chosen square roots and their reciprocals are holomorphic, and
\[
 |q_{j_a}|+|\bar q_{j_a}|\le C\sqrt\nu,\quad
 |\partial_{y_a}q_{j_a}|+|\partial_{y_a}\bar q_{j_a}|
       \le C\nu^{-1/2},\quad
 |\partial_{x_a}q_{j_a}|+|\partial_{x_a}\bar q_{j_a}|
       \le C\sqrt\nu.
\]
In particular $a_J\le C\sqrt\nu$ and $\zeta\le r$.
Applying the chain rule to \cref{lem:normal-degree-bounds} gives
\[
 \begin{aligned}
 |X^x_{P_4^{\ge3}}|&\le Cr^3\nu^{-1/2},\\
 r^{-2}|X^y_{P_4^{\ge3}}|&\le C\sqrt\nu\,r,\\
 r^{-1}(\norm{X^z_{P_4^{\ge3}}}_1
                +\norm{X^{\bar z}_{P_4^{\ge3}}}_1)
     &\le C(\sqrt\nu\,r+r^2).
 \end{aligned}
\]
Together with Step~1 this proves
\[
 \norm{X_{P_{yy}+P_{yz^2}+P_4^{\ge3}}}_r
 \le C(r^2+\sqrt\nu\,r+r^3/\sqrt\nu)
 =C(\nu^{1+2\sigma}+\nu^{1+\sigma}+\nu^{1+3\sigma})
 \le C\nu^{1+\sigma}.
\]

\medskip
\noindent\textit{Step 3: parameter derivatives and the Lipschitz bound.}
Only $P_4^{(3)}$ depends on $\xi$.  Its tangential amplitudes have the
form $(\xi_a+y_a)^{1/2}e^{\pm ix_a}$; their parameter derivatives have
size $C\nu^{-1/2}$.  The $x$-component of the vector field already
contains one $y$-derivative of such an amplitude.  Its additional
parameter derivative therefore has size $C\nu^{-3/2}$.
Using the same cubic normal factors as in Step~2, we obtain
\[
 \begin{aligned}
 |\partial_\xi X^x|&\le Cr^3\nu^{-3/2},\\
 r^{-2}|\partial_\xi X^y|&\le Cr\nu^{-1/2},\\
 r^{-1}(\norm{\partial_\xi X^z}_1
                   +\norm{\partial_\xi X^{\bar z}}_1)
       &\le Cr\nu^{-1/2}.
 \end{aligned}
\]
The derivative norm includes all $b$ parameter directions.  Since
$\Pi_\nu$ is convex, these bounds also control the parameter Lipschitz
seminorm at each fixed phase point.
Consequently
\[
 \norm{X_{P_{yy}+P_{yz^2}+P_4^{\ge3}}}^{\Lip}_r
 \le C(r\nu^{-1/2}+r^3\nu^{-3/2})
 =C(\nu^\sigma+\nu^{3\sigma})\le C\nu^\sigma.
\]
\end{proof}

\begin{lemma}[Sixth-order remainder estimate]\label{lem:R6-estimate}
On $D(s,r_\nu)\times\Pi_\nu$,
\begin{align}
 \norm{X_{R_6^J}}_{r_\nu}
 &\le C_{\omega,J}\nu^{2-2\sigma},
 \label{eq:R6-sup-scaled}\\
 \norm{X_{R_6^J}}^{\Lip}_{r_\nu}
 &\le C_{\omega,J}\nu^{1-2\sigma}.
 \label{eq:R6-lip-scaled}
\end{align}
\end{lemma}

\begin{proof}
In this proof $Q=(q,\bar q)$ denotes the complex spectral coordinates,
and $V(Q)=X_{R_6^J}(Q)$ denotes the vector field before the
action--angle substitution.  Analyticity and vanishing to order five
give, on a fixed sufficiently small ball,
\[
 \norm{V(Q)}\le C\norm Q^5,\qquad
 \norm{DV(Q)}\le C\norm Q^4.
\]
The derivative estimate follows from the Cauchy estimate on a ball
of radius proportional to $\norm Q$; at $Q=0$ it follows from the
vanishing Taylor coefficients.

On the action--angle domain,
$\norm Q\le C(\sqrt\nu+r_\nu)\le C\sqrt\nu$.
We therefore have
\[
 \norm{V(Q)}\le C\nu^{5/2},\qquad
 \norm{DV(Q)}\le C\nu^2.
\]
The vector-field components transform as follows.
For one tangential mode write $q=q_{j_a}$, $\bar q=\bar q_{j_a}$ and
$I=\xi_a+y_a=q\bar q$.  At fixed $\xi$, differentiating
$q=\sqrt I e^{-ix_a}$ and $\bar q=\sqrt I e^{ix_a}$ along a trajectory
and solving for the velocities gives
\[
 X^{x_a}=\frac i2\left(\frac{V^{q_{j_a}}}{q_{j_a}}
                       -\frac{V^{\bar q_{j_a}}}{\bar q_{j_a}}\right),
 \qquad
 X^{y_a}=\bar q_{j_a}V^{q_{j_a}}
                      +q_{j_a}V^{\bar q_{j_a}}.
\]
The normal components are unchanged.  On the fixed strip, the
tangential amplitudes have size at most $C\sqrt\nu$ and their
reciprocals at most $C\nu^{-1/2}$.  Thus
\[
 |X^x|\le C\nu^2,\qquad |X^y|\le C\nu^3,
 \qquad \norm{X^z}_1+\norm{X^{\bar z}}_1\le C\nu^{5/2}.
\]
With $r=r_\nu=\nu^{1/2+\sigma}$, the scaled norm is bounded by
\[
 C\left(\nu^2+\frac{\nu^3}{r^2}+\frac{\nu^{5/2}}r\right)
 =C(\nu^2+\nu^{2-2\sigma}+\nu^{2-\sigma})
 \le C\nu^{2-2\sigma},
\]
which proves \eqref{eq:R6-sup-scaled}.

For the parameter estimate, differentiation is at fixed
$(x,y,z,\bar z)$.  Only the finitely many tangential coordinates of
$Q$ depend on $\xi$, and
$\norm{\partial_{\xi_a}Q}\le C\nu^{-1/2}$.
The chain rule therefore gives
\[
 \norm{\partial_{\xi_a}(V(Q))}
 \le\norm{DV(Q)}\norm{\partial_{\xi_a}Q}
 \le C\nu^{3/2}.
\]
The derivatives of the tangential amplitudes have size
$C\nu^{-1/2}$, and those of their reciprocals have size
$C\nu^{-3/2}$.  Differentiating the two transformation formulas above yields
\[
 \begin{aligned}
 |\partial_\xi X^x|
   &\le C(\nu^{-3/2}\nu^{5/2}
                       +\nu^{-1/2}\nu^{3/2})\le C\nu,\\
 |\partial_\xi X^y|
   &\le C(\nu^{-1/2}\nu^{5/2}
                       +\nu^{1/2}\nu^{3/2})\le C\nu^2,\\
 \norm{\partial_\xi X^z}_1+\norm{\partial_\xi X^{\bar z}}_1
   &\le C\nu^{3/2}.
 \end{aligned}
\]
Integrating these bounds on segments in the convex parameter box gives
\[
 \norm{X_{R_6^J}}^{\Lip}_r
 \le C\left(\nu+\frac{\nu^2}{r^2}+\frac{\nu^{3/2}}r\right)
 =C(\nu+\nu^{1-2\sigma}+\nu^{1-\sigma})
 \le C\nu^{1-2\sigma}.
\]
\end{proof}

\begin{proposition}[Verification of (P3) and smallness]\label{prop:P3}
For $p=\bar p=1$, the perturbation \eqref{eq:Pnu} satisfies (P3) and
\begin{align}
 \norm{X_{P_\nu}}_{r_\nu}
 &\le C_{\omega,J}\nu^{1+\sigma},
 \label{eq:Pnu-sup}\\
 \norm{X_{P_\nu}}^{\Lip}_{r_\nu}
 &\le C_{\omega,J}\nu^\sigma.
 \label{eq:Pnu-lip}
\end{align}
Moreover,
\begin{equation}\label{eq:smallness-ratio}
 \frac1{\alpha_\nu}
 \left(
 \norm{X_{P_\nu}}_{r_\nu}
 +\frac{\alpha_\nu}{\mathfrak M}
 \norm{X_{P_\nu}}^{\Lip}_{r_\nu}
 \right)
 \longrightarrow0
 \qquad(\nu\downarrow0).
\end{equation}
\end{proposition}

\begin{proof}
By \cref{prop:P4-analytic,prop:partial-BNF,lem:normal-degree-bounds},
the spectral-coordinate perturbations have analytic vector fields
in $h^1\times h^1$.  The action--angle substitution is holomorphic on
the chosen domains, and the preceding two lemmata give Lipschitz
dependence on $\xi$.  This verifies (P3) with $p=\bar p=1$.

By \cref{lem:quartic-estimate,lem:R6-estimate},
\[
 \begin{aligned}
 \norm{X_{P_\nu}}_{r_\nu}
 &\le C(\nu^{1+\sigma}+\nu^{2-2\sigma}),\\
 \norm{X_{P_\nu}}^{\Lip}_{r_\nu}
 &\le C(\nu^\sigma+\nu^{1-2\sigma}).
 \end{aligned}
\]
Since $\sigma<1/3$, the remainder terms have strictly larger exponents,
which proves \eqref{eq:Pnu-sup}--\eqref{eq:Pnu-lip} for $0<\nu<1$.

Finally $\alpha_\nu=\nu^{1+\beta}$, and hence
\[
 \begin{aligned}
 \frac1{\alpha_\nu}
 \left(\norm{X_{P_\nu}}_{r_\nu}
    +\frac{\alpha_\nu}{\mathfrak M}
              \norm{X_{P_\nu}}^{\Lip}_{r_\nu}\right)
 &\le C\nu^{\sigma-\beta}
       +\frac C{\mathfrak M}\nu^\sigma\longrightarrow0.
 \end{aligned}
\]
Here both exponents are positive because $0<\beta<\sigma$.
\end{proof}

\begin{lemma}[Weak symplectic conjugacy]\label{lem:weak-conjugacy}
Let $\Psi$ be a $C^1$ symplectic diffeomorphism between open subsets of
$h^1\times h^1$, and let $K=H\circ\Psi$.  Suppose that
$p\in C^1(\R;h^1\times h^1)$ satisfies
\begin{equation}\label{eq:weak-Hamilton-K}
 \Om_{\mathrm{symp}}(\dot p(t),V)
 =\dd K(p(t))[V]
 \qquad\text{for every }V\in h^1\times h^1.
\end{equation}
Then $q=\Psi\circ p$ satisfies
\begin{equation}\label{eq:weak-Hamilton-H}
 \Om_{\mathrm{symp}}(\dot q(t),W)
 =\dd H(q(t))[W]
 \qquad\text{for every }W\in h^1\times h^1.
\end{equation}
This statement only uses the continuity of the differential of $H$; it
does not require the formal vector field $X_H$ to take values in $h^1$.
\end{lemma}

\begin{proof}
The chain rule gives $\dot q=D\Psi(p)\dot p$.  For an arbitrary
$W\in h^1\times h^1$, set $V=D\Psi(p)^{-1}W$.  Symplecticity and
\eqref{eq:weak-Hamilton-K} yield
\[
 \Om_{\mathrm{symp}}(\dot q,W)
 =\Om_{\mathrm{symp}}(\dot p,V)
 =\dd K(p)[V]
 =\dd H(q)[D\Psi(p)V]
 =\dd H(q)[W].
\]
This is \eqref{eq:weak-Hamilton-H}.
\end{proof}

\begin{proof}[Proof of \cref{thm:main}]
\noindent\textit{Step 1: choice of the path and fixed constants.}
Set
\[
 \Omega_*=\Omega_{\mathrm{sp}}\cap\Omega_{\mathrm{nr}}
                         \cap\Omega_{\mathrm{tw}}.
\]
The preceding results show that $\Omega_*$ is measurable, has
probability one, and is independent of $J$.

Fix $\omega\in\Omega_*$, a finite nonempty $J$, and $0<\beta<1/3$.
Choose $\beta<\sigma<1/3$ and $0<\delta<1$; all these choices remain
fixed as $\nu$ decreases.  Fix also a strip width $s>0$ and a
Diophantine exponent $\tau\ge b+3$, which satisfies the geometric
summability requirement for $d=2$.  Construct the partial Birkhoff
map $\Phi_J^1$, the reference parameter box $\Pi^0$, and the frequency
bounds $\mathfrak M,\mathfrak L$ as above.

\medskip
\noindent\textit{Step 2: application of the abstract KAM theorem.}
Use the scales \eqref{eq:scales}.  For sufficiently small $\nu$, the
action--angle image of $D(s,r_\nu)$ lies in the domain of the Birkhoff
map, and the transformed Hamiltonian is $N_\xi+P_\nu$ up to an additive
constant.  By \cref{prop:P1,prop:P2,prop:P3}, the hypotheses of
\cref{thm:Poschel} hold with the fixed bounds \eqref{eq:correct-M-L}.
The uniformity lemma and \eqref{eq:smallness-ratio} give a threshold
$\nu_0=\nu_0(\omega,J,\beta)>0$ below which the KAM smallness condition
holds.

For every $0<\nu<\nu_0$, the KAM theorem supplies a Cantor parameter
set $\Pi_\nu^*\subset\Pi_\nu$ and, for each $\xi\in\Pi_\nu^*$, a real
analytic torus embedding $\iota_\xi$ in action--angle and normal
coordinates.  Its trajectories are
$\iota_\xi(\theta_0+\omega_*(\xi)t)$, where $\omega_*(\xi)$ is the
resulting Diophantine frequency vector.

\medskip
\noindent\textit{Step 3: measure estimate for the excluded actions.}
The frequency maps are affine and $d=2>1$, so the uniform geometric
estimate gives
\[
 |\Pi_\nu\setminus\Pi_\nu^*|
 \le C_{\omega,J}(\operatorname{diam}\Pi_\nu)^{b-1}\alpha_\nu
 \le C_{\omega,J}\nu^{b-1}\nu^{1+\beta}
 =C_{\omega,J}\nu^{b+\beta}.
\]
Since $\operatorname{diam}\Pi_\nu=\sqrt b\,\nu$ and
$|\Pi_\nu|=\nu^b$, division by $|\Pi_\nu|$ proves
\eqref{eq:relative-measure}.

\medskip
\noindent\textit{Step 4: conjugation to the original Hamiltonian.}
Let $\mathcal A_\xi$ be the action--angle map
\eqref{eq:action-angle}.  The calculation
\eqref{eq:AA-symplectic} shows that this map is symplectic in local
angle charts.  Since $\iota_\xi$ is analytic and its angle space is
finite dimensional, the curve
\[
 p(t)=\mathcal A_\xi\bigl(
                  \iota_\xi(\theta_0+\omega_*(\xi)t)\bigr)
\]
is $C^1$ with values in $h^1\times h^1$.  The KAM invariance equation,
paired with arbitrary tangent vectors and transferred through
$\mathcal A_\xi$, is precisely the weak Hamiltonian identity for
$\widetilde H=H\circ\Phi_J^1$.
Apply \cref{lem:weak-conjugacy} with $\Psi=\Phi_J^1$ and
$K=\widetilde H$.  It follows that
\[
 \theta\longmapsto
 \Phi_J^1\circ\mathcal A_\xi\circ\iota_\xi(\theta)
\]
is an invariant torus for the original spectral Hamiltonian $H$.
Its image under the bounded spectral transform belongs to
$H^1_0(0,\pi;\C)$, and the embedding is real analytic.

The smallness ratio in \eqref{eq:smallness-ratio} tends to zero, so the
KAM embedding estimates give a deformation tending to zero in the
scaled phase norm relative to the trivial torus
$(x,y,z,\bar z)=(\theta,0,0,0)$.  Its action--angle image is the linear
rotational torus with tangential amplitudes $\sqrt{\xi_a}$.
The Birkhoff map changes a spectral point of size $O(\sqrt\nu)$ by
$O(\nu^{3/2})$, by Step~1 of the normal-form proof.  Thus the resulting
torus is a small deformation of the linear torus asserted in the theorem.
All transformations preserve the conjugacy relation $\bar q=\overline q$.
P\"oschel's theorem gives an elliptic normal linearized system at each
torus.  Bounded real symplectic conjugacies preserve this linear
stability: along the compact torus their differentials and inverse
differentials are uniformly bounded and conjugate the corresponding
normal variational equations.

\medskip
\noindent\textit{Step 5: the PDE and its mild formulation.}
For the fixed path the form estimate defines a bounded map
\[
 A_\omega:H^1_0(0,\pi)\longrightarrow H^{-1}(0,\pi),
 \qquad
 \langle A_\omega u,v\rangle_{H^{-1},H^1_0}
     =\mathfrak a_{\sigma_\omega}[u,v].
\]
The one-dimensional product estimate also gives a continuous real analytic
map $u\mapsto|u|^2u$ from $H^1_0(0,\pi;\C)$ to itself, regarded as a
map between real Banach spaces.
Writing $q(t)=\Phi_J^1(p(t))$ with $\bar q(t)=\overline{q(t)}$ and
$u(t)=\sum_nq_n(t)\phi_n$, the weak Hamiltonian identity, tested in
individual spectral coordinates, gives
\[
 i\dot q_n(t)=\lambda_nq_n(t)
       +\kappa\int_0^\pi |u(t,x)|^2u(t,x)\phi_n(x)\dd x.
\]
Finite spectral sums are dense in $H^1_0$.  The boundedness of
$A_\omega:H^1_0\to H^{-1}$ and of the nonlinear term therefore
extends this coordinate identity to
\[
 iu_t=A_\omega u+\kappa|u|^2u
 \quad\text{in }H^{-1}(0,\pi).
\]
In particular the trajectories have the claimed regularity
$u\in C(\R;H^1_0)\cap C^1(\R;H^{-1})$.

For completeness, put $U(t)=e^{-itA_\omega}$.  In the shifted spectral
energy norm the factors $e^{-it\lambda_n}$ have modulus one, so $U(t)$
is a strongly continuous group preserving $H^1_0$ and extending to its
dual $H^{-1}$.  Variation of constants in each coordinate gives
\[
 u(t)=U(t)u(0)-i\kappa\int_0^tU(t-s)(|u(s)|^2u(s))\dd s.
\]
The integral is a Bochner integral in $H^1_0$, because its integrand
is continuous in that space.  Taking spectral coefficients proves
that this expression agrees with the trajectory just constructed.
Conversely, this mild identity differentiates in $H^{-1}$: the
generator acts boundedly from $H^1_0$ to $H^{-1}$, and the nonlinear
term is continuous in $H^1_0$.  It yields the displayed PDE and the
stated time regularity.  This proves the equivalence and completes
the theorem.
\end{proof}

\section{Discussion}

We discuss the path dependence of the KAM bounds and the additional
estimates required if the Brownian path is used as an external parameter.

\paragraph{Pathwise amplitude thresholds.}
For fixed $J$ and $\beta$, the conclusion has the quantifier order
\[
 \text{for almost every }\omega,\quad
 \text{there exists }\nu_0(\omega,J,\beta)>0,\quad
 \text{for every }0<\nu<\nu_0(\omega,J,\beta).
\]
The threshold depends on the spectral estimates, inverse twist, and
normal-form bounds.  The spectral-to-physical norm equivalence also
depends on the path, so actions of size $\nu$ give tori of size
$O_{\omega,J,\beta}(\sqrt\nu)$ in $H^1_0$.  Neither the threshold nor
the embedding bounds are asserted to be uniform in $\omega$ or $J$.

\paragraph{Uniform bounds and high-probability restrictions.}
There is a concrete obstruction to imposing a deterministic positive
nonresonance constant on a full-probability set.  Fix $n_0\notin J$ and
consider the deterministic path
\[
 B_*(x)=-\frac{n_0^2}{\rho}x.
\]
Its potential is the constant $-n_0^2$, so
$\lambda_{n_0}(B_*)=0$.  Classical Wiener measure has full support in
$E$; this follows from the density of its Cameron--Martin space and the
Gaussian support theorem, see~\cite{Bogachev1998}.  Continuity of the
ordered eigenvalues consequently implies
\[
 \mu\{B:|\lambda_{n_0}(B)|<\gamma\}>0
 \qquad\text{for every }\gamma>0.
\]
Thus the gap in \cref{lem:global-two-mode-gap} has no deterministic
positive lower bound on a full-probability set.  Fixed nonresonance
constants require a positive-probability exclusion.  This obstructs
uniform application of the present estimates, without implying
nonexistence of invariant tori near the excluded resonances.

A weaker uniform formulation is available on sets of arbitrarily high
probability.  For fixed $J$ and $\beta$, choose the positive threshold
$\nu_0$ and the finite constants in the estimates measurably, as is
possible from the spectral formulas and the countable suprema used in
the proof.  Restrict to increasing measurable sets on which $\nu_0$ is
bounded below and those constants are bounded above.  These sets exhaust
$\Omega_*$ up to a null set.  Hence, for every $\eta>0$, deterministic
thresholds and bounds can be chosen on a set of probability at least
$1-\eta$.  This observation gives no rate relating the admissible
amplitude to $\eta$.  Such a rate would require quantitative control of
the distributions of the path-dependent constants.

\paragraph{The Brownian path as an external parameter.}
An alternative construction would prescribe the tangential actions and
use variation of $B$ to adjust the frequencies.  The parameter exclusion
would then take place in Wiener space rather than action space.  At the
linear level the relevant resonance functions are
\[
 \mathcal R_{k,l}(B)
 =\sum_{a=1}^b k_a\lambda_{j_a}(B)
   +\sum_{n\notin J}l_n\lambda_n(B),
 \qquad (k,l)\ne0,\quad |l|\le2.
\]
At a fixed positive perturbation scale, such a construction would seek
a positive- or high-probability set on which quantitative nonresonance
conditions hold throughout the iteration.  The zero-set theorem alone
does not establish a full-probability result at that scale.

An almost-sure result with path-dependent thresholds is a different
question.  Sets whose probabilities approach one as the amplitude
vanishes do not alone yield an almost-sure statement valid at every
sufficiently small amplitude; additional control across scales is needed.

\paragraph{Quantitative sublevel estimates.}
To estimate exclusions in Wiener space, one needs bounds for small
values of resonance functions, not only the nullity of their zero sets.
For example, on a measurable set $G$ where the required spectral norms
are controlled, one could seek estimates of the form
\[
 \mu\bigl(G\cap\{B:|\mathcal R_{k,l}(B)|<t\}\bigr)
 \le C_{G,k,l}\,t^{a_{G,k,l}},\qquad 0<t<1,
\]
with positive exponents and sufficiently controlled dependence on the
resonance labels.  Separate estimates for each label are not enough:
after using the high-energy spectral bounds, the remaining exclusions
must be summable over the Fourier and normal indices and over the KAM
steps.  The estimates must also remain valid for the perturbed frequency
maps, which need not retain the analytic dependence of the original
spectral data.

The Feynman--Hellmann formula provides a starting point for such an
analysis:
\[
 D_{H_g}\mathcal R_{k,l}(B)
 =\rho\int_0^\pi g(x)
   \left(\sum_{a=1}^b k_a\phi_{j_a}(B,x)^2
        +\sum_{n\notin J}l_n\phi_n(B,x)^2\right)\dd x.
\]
A possible approach is to vary finitely many Cameron--Martin coordinates
conditionally on the others.  It would require quantitative invertibility
of the tangential frequency map, transversality bounds for mixed
resonances with controlled index dependence, and stability under the
KAM frequency corrections.  These estimates are not provided by the
qualitative zero-set argument.

\end{document}